\documentclass[english, 10pt, a4paper]{article}
\PassOptionsToPackage{dvipsnames}{xcolor}
\usepackage{multirow}
\usepackage[hyphens]{url}
\usepackage{lipsum}
\usepackage{amsfonts}
\usepackage{graphicx}
\usepackage{epstopdf}
\usepackage[toc,page]{appendix}
\usepackage{algorithmic}
\ifpdf
  \DeclareGraphicsExtensions{.eps,.pdf,.png,.jpg}
\else
  \DeclareGraphicsExtensions{.eps}
\fi

\usepackage{enumitem}
\setlist[enumerate]{leftmargin=.5in}
\setlist[itemize]{leftmargin=.5in}

\PassOptionsToPackage{dvipsnames}{xcolor}
\usepackage{babel}
\usepackage[utf8]{inputenc}
\usepackage[T1]{fontenc}
\usepackage[top=2.5cm, bottom=2.5cm, left=2.5cm, right=2.5cm]{geometry} % marges
\usepackage{verbatim}
\usepackage{amsmath}
\usepackage{mathtools}
\mathtoolsset{showonlyrefs}
\usepackage{mwe,tikz}
\usepackage{amsthm}
\usepackage{amsfonts}
\usepackage{mathabx}
\usepackage{color}
\usepackage{xcolor}
\usepackage{amssymb}
\usepackage{mathrsfs}
\usepackage[ruled]{algorithm2e}
\usepackage{subfigure}
\usepackage{float}
\usepackage[most]{tcolorbox}
\usepackage{xr-hyper ,hyperref}
\usepackage{xspace}
\usepackage{regexpatch}
\usepackage{colortbl}%
\usepackage{makecell}
\usepackage{microtype}
\usepackage{pgfplots}
\usepackage{graphicx}
\usepackage{soul}
\usepackage{calc}
\usepackage{cleveref}
\usepackage{graphicx}

\DeclareMathAlphabet{\mathb}{OML}{cmm}{b}{it}
\newcommand{\N}{\mathbb{N}}		% N double barre (entiers)
\newcommand{\R}{\mathbb{R}}		% R double barre (réels)
\renewcommand{\H}{\mathcal{H}}
\newcommand{\X}{\mathcal{X}}

\newcommand\innerprod[2]{\left\langle #1\,|\, #2 \right\rangle}
\newcommand{\fonc}[3]{#1:  #2  \rightarrow  #3}					% fonction sans 2ème ligne
\newcommand{\syst}[1]{\left \{ \begin{array}{l} #1 \end{array} \right. \kern-\nulldelimiterspace}	% système
\newcommand{\prox}{\text{\normalfont prox}}
\newcommand{\dist}{\operatorname{dist}}

\DeclareMathOperator*{\argmin}{argmin}
\newcommand{\dom}{\text{\normalfont dom}\,}

\newcommand{\minimize}[2]{\ensuremath{\underset{\substack{{#1}}}{\mathrm{minimize}}\;\;#2 }}

\newcommand{\proj}{\ensuremath{\text{\rm proj}}}

\makeatletter
\newlength{\algorithmboxrule}
\newcommand{\algorithmboxcolor}{white}
\xpatchcmd*{\algocf@caption@boxruled}{0.0pt}{2\algorithmboxrule}{}{}
\xpatchcmd*{\algocf@caption@boxruled}{\vrule}{\vrule width \algorithmboxrule}{}{}
\xpatchcmd{\algocf@caption@boxruled}{\hrule}{\hrule height \algorithmboxrule}{}{}
\xpretocmd{\algocf@caption@boxruled}{\color{\algorithmboxcolor}}{}{}
\makeatother

\newtheorem{proposition}{Proposition}
\newtheorem{lemma}{Lemma}
\newtheorem{theorem}{Theorem}
\newtheorem{corollary}{Corollary}
\newtheorem{example}{Example}
\theoremstyle{definition}
\newtheorem{definition}{Definition}
\newtheorem{remark}{Remark}
\newtheorem{assumption}{Assumption}

\allowdisplaybreaks
\definecolor{shadecolor}{rgb}{0.8,0.8,0.8}

\title{Forward-Backward algorithms for weakly convex problems}

\author{Ewa Bednarczuk\thanks{Warsaw University of Technology,  00-662 Warsaw, Koszykowa 75, Poland} \thanks{Systems Research Institute, PAS,  01-447 Warsaw, Newelska 6, Poland}
\and Giovanni Bruccola\footnotemark[2] \and Gabriele Scrivanti\thanks{Université Paris-Saclay, Inria, CentraleSupélec, CVN, 3 Rue Joliot Curie, 91190, Gif-Sur-Yvette, France.} 
\and The Hung Tran\footnotemark[2]}

\usepackage{amsopn}

\date{}
\begin{document}
\maketitle

\begin{abstract} 
We investigate the convergence properties of exact and inexact forward-backward algorithms to minimise the sum of two weakly convex functions defined on a Hilbert space, where one has a Lipschitz-continuous gradient. We show that the exact forward-backward algorithm converges strongly to a global solution, provided that the objective function satisfies a sharpness condition. For the inexact forward-backward algorithm, the same condition ensures that the distance from the iterates to the solution set approaches a positive threshold depending on the accuracy level of the proximal computations. As an application of the considered setting, we provide numerical experiments related to  discrete tomography.
\end{abstract}
\vspace{0.4cm}


% REQUIRED

\section{Scope of this Work}

Given a Hilbert space $\X$, we consider a problem of the form
\begin{equation}
\label{eq: ourproblem}
    \minimize{x\in\X}{f(x) + g(x)}
\end{equation}
where function $\fonc{f}{\X}{(-\infty,+\infty]}$ is weakly convex with modulus $\rho_f\geq 0$ ($\rho_f$-weakly convex) and ${\fonc{g}{\X}{\R}}$ is convex and {Fr\'echet} differentiable with a $L_g$-Lipschitz continuous {(Fr\'echet) gradient $\nabla g$}. 

We highlight that assuming convexity of $g$ in \eqref{eq: ourproblem} is not restrictive: 
whenever a  problem is to minimize the sum of a function $\fonc{F}{\X}{(-\infty,+\infty]}$ that is $\rho_F$-weakly convex and a function ${\fonc{G}{\X}{\R}}$ that is {Fr\'echet} 
 differentiable with a $L_G$-Lipschitz continuous (Fr\'echet) gradient $\nabla G$, 
 we can 
 adopt the following reformulation
 \begin{equation}
\label{eq: ourproblem0}
\begin{split}
    %(\forall x \in \X)\quad \quad F(x) + G(x) = 
    %{f(x) -\frac{\rho_g}{2}\|x\|^2+\frac{\rho_g}{2}\|x\|^2+\ g(x)}\\
 \minimize{x\in\X}  \underbrace{ {\left(F(x) -\frac{L_G}{2}\|x\|^2\right)}}_{f(x)}+\underbrace{\left(\frac{L_G}{2}\|x\|^2+\ G(x)\right)}_{g(x)},
    %& \minimize{x\in\X}{F(x) +G(x),\ \ \text{ }F(x):=f(x) -\frac{\rho_g}{2}\|x\|^2, G(x):=\frac{\rho_g}{2}\|x\|^2+\ g(x)}
   % = f(x) +  g(x)
\end{split}
\end{equation}
where $\fonc{f}{\X}{(-\infty,+\infty]}$ is $\rho_f$-weakly convex  with $\rho_f=\left(\rho_F+L_G\right)$, $\fonc{ g}{\X}{\R}$ is convex and differentiable with a $L_g$-Lipschitz continuous gradient with $L_g=2L_G$. This stems from the fact that function $G$ is  $L_G$-weakly convex (see \Cref{prop:subdiff}-c). 
%In conclusion, we retrieve a problem of the form presented in \eqref{eq: ourproblem}.

 From a theoretical viewpoint, weakly convex functions have been investigated on different levels of generality, \emph{e.g.\ } \cite{jourani1996Subdifferentiability, rolewicz1993globalization, vial1983strong}. In applications, weakly convex functions appear \emph{e.g.\ }in robust phase retrieval and matrix factorization. More examples can be found in  \cite{Davis2018SubgradientMethodsSharpWeaklyConvex}.  Structured problems of the form presented in \eqref{eq: ourproblem} frequently appear in variational sparse signal and image processing, where function $g$ corresponds to the data-fidelity term and function $f$ is a weakly convex penalty function approximating the $\ell_0$ pseudonorm \cite{fan2001variable,
soubies2015continuous,zhang2010nearly}.
In the finite-dimensional setting, there exist algorithms which are devoted to solving problems that involve weakly convex functions, see \emph{e.g.\ }\cite{Atenas2023Unified, bayram2015convergence, davis2019stochastic,Davis2018SubgradientMethodsSharpWeaklyConvex}. {In particular, the authors in \cite{bayram2015convergence} analyse the convergence of the FB splitting under the assumption that the smooth function $g$ is $\rho$-strongly convex, implying that the overall objective $f+g$ is convex.}

The main scope of this paper is to analyse the convergence properties of the  following inexact forward-backward splitting {when applied to problem \eqref{eq: ourproblem}}, with step size $\alpha_t>0$ and accuracy level $\varepsilon_t\ge0$
    % \begin{equation}
    % \label{eq: FB upate}
    % x_{t+1} = \varepsilon_{t}\text{-}\prox_{\alpha_t f}(x_t - \alpha_t \nabla g(x_t))
    % \end{equation}
    % or
     \begin{equation}
    \label{eq: FB upate : 2}
    x_{t+1} \in \varepsilon_{t}\text{-}\prox_{\alpha_t f}(x_t - \alpha_t \nabla g(x_t)).
    \end{equation}
    
{In order to investigate the strong convergence of the scheme  in  \eqref{eq: 
FB upate : 2} (see \Cref{theo:Et constant convergence} and \Cref{prop:convergence_convex:new}), we exploit a sharpness condition (see \Cref{def:sharpness} below}). There has been a rising number of works which use sharpness to discuss the global convergence of splitting algorithms, such as \cite{Li2021_WeaklyCOnvexOptimization,li2019incremental}. 
We take into account a possibly inexact proximal step for the weakly convex function (the exact proximal step is considered in {\cite[Section 5]{bohm2021variable}} to infer a complexity bound for the exact forward-backward algorithm in a similar framework). {We use the concept of $\varepsilon$-proximal operator introduced in \cite{bednarczuk2022calculus} (see \Cref{def:eps_prox} below). Its properties, expressed in terms of global proximal ${\varepsilon\text{-subdifferential}}$s, are investigated in {\cite[Proposition 4, Proposition 5]{bednarczuk2022calculus}}, where a comparison with other concepts of $\varepsilon$-proximal operator in the convex case is provided.
The concept of an inexact proximal operator and its computational tractability lie at the core of many proximal methods referring to fully convex case, }where both $f$ and $g$ are convex \cite{barre2022principled,Millan2019InexactProximal,  salzo2012inexact}.\\

The main contribution of this work can be summarised in the following points.   \begin{enumerate}
 \item We exploit the proximal subdifferentials -- which are defined globally for the class of weakly convex functions -- to analyse the convergence of \eqref{eq: FB upate : 2} to global minima for the exact case ($\varepsilon=0$) and, for the inexact case, with a constant inexactness level ($\varepsilon>0$). 

 \item The convergence  properties of \eqref{eq: FB upate : 2} to global minima, are investigated  under the assumption that the objective function is sharp, Definition \ref{def:sharpness}. An important consequence of  sharpness of the objective  is that,  in the exact case, the iterates converge to global solution  of \eqref{eq: ourproblem} provided the starting point is sufficiently close to the solution set, Proposition \ref{prop:convergence_convex:new}.
 %of important information about the area around the global minimizers, which does not contain any critical points. 
 % \item We take into account a possibly inexact proximal step for the weakly convex function. The exact proximal step is also considered in {\cite[Section 5]{bohm2021variable}} to infer a complexity bound for the exact forward-backward algorithm. {We use the concept of $\varepsilon$-proximal operator introduced in \cite{bednarczuk2022calculus}, see also \Cref{def:eps_prox} below. Its properties, expressed in terms of global proximal ${\varepsilon\text{-subdifferential}}$s, are investigated in {\cite[Proposition 4, Proposition 5]{bednarczuk2022calculus}}, where a comparison with other concepts of $\varepsilon$-proximal operator in the convex case is provided. The concept of an inexact proximal operator and its computational tractability lie at the core of many proximal methods referring to fully convex case, }where both $F$ and $G$ are convex \cite{barre2022principled,Millan2019InexactProximal,  salzo2012inexact}.
\item We provide a supporting application on binary tomography to demonstrate the performance of  \eqref{eq: FB upate : 2}.\\
% We work on Hilbert spaces instead of $\mathbb{R}^n$. The convergence results in the previous points of this list should be understood in terms of strong convergence.

\end{enumerate}

{All the results are formulated in a Hilbert space. However, we are not using any dimensionality-dependent arguments.  The structure of the underlying space is exploited exclusively via calculus rules of the  inner product and the respective norm. Consequently, the results remain valid without any change in  Euclidean spaces. For applications of the considered setting in Hilbert spaces, see e.g. \cite{azmi2023nonmonotone}.}

\subsection{Related Works}
\paragraph{The FB Algorithm}The FB algorithm (or proximal gradient algorithm) belongs to the class of splitting methods \cite{han1988parallel,lions1979splitting} whose aim is to minimise the sum of a smooth function and a non-smooth one. 
%By taking a gradient step on the smooth function and the proximal step on the non-smooth one, it is only needed to access each function separately. 
This kind of algorithm has been well studied and understood in Hilbert spaces for the convex case \cite{han1988parallel, combettes2005signal}. 
Many variants of FB have been proposed recently \emph{e.g.\ }inertial FB \cite{bello2022fista,lorenz2015inertial,bonettini2023abstract}, to accelerate the algorithm, variable metric FB  \cite{combettes2014variable} to improve convergence, or FB in Banach spaces, \cite{guan2022forward}.

When the convexity hypothesis is relaxed, two main drawbacks arise. Firstly, the convergence to global minimiser cannot be easily guaranteed any longer. Secondly, the lack of convexity of the non-smooth function invalidates the fact that the corresponding proximal operator is single-valued.
 In particular, in \cite{Bonettini_2017}, a line search-based inexact proximal gradient scheme is designed for problems where function $g$ is not necessarily convex, whereas $f$ is convex and its proximal map can be approximated using the strategy proposed in \cite{salzo2012inexact}. Therefore, we highlight that in 
\Cref{alg:epsFB} we take the proximal step on the weakly convex function $f$, while the function $g$ is Fr\'echet differentiable with a $L_g$-Lipschitz continuous gradient and also convex, by the manipulation in \eqref{eq: ourproblem0}.\\

\paragraph{Convergence under error bound conditions} In the general non-convex case, the seminal works \cite{attouch2013convergence, liu2017further} illustrate that the convergence to a local minimum can be ensured provided that the objective function satisfies the Kurdyka-{\L}ojasiewicz (KL) inequality \cite{Kurdyka1998GradientsOfFunctions,Lojasiewicz1963ProprieteTopologique,Lojasiewicz1993GeometrieSemiEtSousAnalytique}. 
 The attractiveness of the KL inequality comes from the fact that, in finite-dimensional spaces, it is satisfied by, among others, subanalytic and semialgebraic functions, which are wide classes covering both convex and non-convex functions appearing in most applications.
  Given its considerable impact on several fields of applied mathematics, the authors in \cite{bolte2010characterizations} proposed a characterisation of the KL property in a non-smooth and infinite dimensional setting.

%In \Cref{sec:sharpness}, we discuss the relations between sharpness and KL in general Hilbert spaces.
% In  finite-dimensional settings, \cite{Bai2022Equivalence} shows the equivalence, for weakly convex functions, between local sharpness, KL inequality and subdifferential error bound around a given critical point. 
The KL inequality is a sufficient condition for non-linear error bounds conditions, in the sense of 
sufficient (and necessary) conditions for error bounds are discussed in many papers, see e.g. \cite{cuong2022error}, or 
 \cite{Bai2022Equivalence}  in the finite-dimensional case.

In a recent work \cite{Atenas2023Unified}, the authors assume the subdifferential error bound to obtain the convergence of a FB scheme for the sum of a convex and a smooth functions to a stationary point.
Meanwhile, in our setting of a general Hilbert space, we consider the sharpness condition with respect to the set of global minimisers of the sum of a weakly convex and a smooth function. Many innovative contributions of forward-backward splitting methods in non-convex settings, such as \cite{attouch2013convergence, themelis2018forward, li2015global}, study global convergence to a critical point of $f+g$ by using limiting subdifferentials. However, it has already been observed that proximal subgradient is a suitably powerful tool for analysing non-differentiable weakly convex functions \cite{clarke2008nonsmooth,RockWets1998Variational,jourani1996Subdifferentiability}.\\

\paragraph{Proximal Subdifferential}  
     Focusing on the class of weakly convex functions allows us to adopt a specific form of the general Fréchet subdifferential, namely
     %which is known as 
     \textit{proximal subdifferential},  %There exists a vast literature devoted to proximal subdifferentials, 
     see {\em e.g.}\  in the finite-dimensional case, the monograph by Rockafellar and Wets \cite{RockWets1998Variational}, in Hilbert spaces the work by Bernard and Thibault \cite{Bernard2005_ProxRegular} and the book by Clarke \textit{et al.}\ \cite{clarke2008nonsmooth}. 
    In particular, we make use of the fact that weakly convex functions 
    %(among others) 
    enjoy a so-called \emph{globalisation property} \cite{jourani1996Subdifferentiability, rolewicz1993globalization}, which states that the proximal subgradient inequality holds globally (see \Cref{def:paraconvexity} and \Cref{def:eps:prox_subdiff}). 
    %Finally, we highlight that for $\rho$-weakly convex functions, the proximal subdifferential coincides with the Clarke subdifferential (see {\cite[Theorem 3.1]{jourani1996Subdifferentiability}} ).
    \\

\paragraph{Outline} The work is organised as follows.
In \Cref{sec:Preliminaries}, we provide some preliminary facts on weakly convex functions and proximal $\varepsilon$-subdifferentials. In \Cref{sec:inexact}, we show the relation between proximal operators and proximal subdifferentials of weakly convex functions. In \Cref{sec:sharpness}, we introduce the concept of sharpness and investigate its relation with critical points and with points satisfying the KL property for weakly convex functions. In 
    \Cref{sec:inexact FB convergence} we present in \Cref{alg:epsFB} the inexact FB algorithm and prove the descending behaviour of the objective values for iterates generated by \Cref{alg:epsFB}. In \Cref{sec:convergence:analysis}, we show the monotonic decrease of the distance from the iterates to the solution set in the case of inexact proximal computations (see \Cref{theo:Et constant convergence}) and, in \Cref{sec:convergence:exact}, we show the local strong convergence to a global solution of the whole sequence with linear rate in the case of exact proximal computations (see \Cref{prop:convergence_convex:new}).
In \Cref{sec:feasibility}, we apply the considered setting to model binary-constrained problems and report on numerical simulations for the proposed FB scheme in the context of binary tomography.

\section{Preliminaries. Basic Definitions and Notation}\label{sec:Preliminaries}

{ \sloppy In this work, the notation $B(x,\delta)$ denotes the open ball with center in $x\in\X$ and radius $\delta>0$ ($\overline B(x,\delta)$ denotes its closure). For any 
%$x\in\X$, ${B}(x,\delta)$ denotes the open ball with centre  $x$ and radius $\delta>0$ and ${\overline{B}(x,\delta)}$ its closure, while for a 
set $C\subset \X$, and  $\delta>0$, we have $B(C,\delta) = \bigcup_{x\in C} B(x,\delta)$ and $\overline{B}(C,\delta)$ its closure. For a function ${\fonc{h}{\X}{\R}}$ and given real values $a, b \in \R$ we use the notation $[ a \le h \le b]$ to denote the level set ${[ a \le h \le b] = \{x\in\X \mid a\leq  h(x) \leq b\} }$. For a function ${\fonc{h}{\X}{(-\infty,+\infty]}}$, the effective domain, $\dom h$,  is defined as ${\dom h := \{x\in \X\,|\, h(x) < +\infty\}}$ and  we  say that $h$ is \textit{proper} whenever $\dom h \neq \emptyset $. For a set set-valued mapping $M: \X \rightrightarrows \X$, 
%we  use the notation $\dom M$ to indicate the set
% \begin{equation*}
   ${ \dom M := \{x\in \X\,|\, M(x) \neq \emptyset\}.}$\\
% \end{equation*}
}

 We start with the  definition of $\rho$-weak convexity (also known as $\rho$-paraconvexity, see \cite{jourani1996Subdifferentiability,Rolewicz1979paraconvex} or $\rho$-semiconvexity, see \cite{cannarsa2004semiconcave}).

\begin{definition}[\textit{$\rho$-weak convexity}]\label{def:paraconvexity}
Let $\mathcal{X}$ be a Hilbert space. A function $\fonc{h}{\mathcal{X}}{(-\infty,+\infty]}$ is 
%said to be 
$\rho$-weakly convex if there exists a constant $\rho\ge0$ such that for $\lambda\in[0,1]$ the following inequality holds:
\begin{equation}
 (\forall (x,y)\in \mathcal{X}^2) \quad h(\lambda x + (1-\lambda)y) \leq  \lambda h(x) + (1-\lambda) h(y) + \lambda(1-\lambda)\frac{\rho}{2}\|x-y\|^{2}.
\end{equation}
We refer to $\rho$ as the \textit{modulus of weak convexity} of the function $h$. 
\end{definition}

{According to {\cite[Proposition 1.1.3]{cannarsa2004semiconcave}}, when $\mathcal{X}$ is a Hilbert space, a $\rho$-weakly convex function $\fonc{h}{\mathcal{X}}{(-\infty,+\infty]}$  is equivalently characterised by the fact that $h + \frac{\rho}{2}\|\cdot\|^2$ is a convex function.}
 Further characterisations of weak convexity can be found in \cite{bednarczuk2022calculus} and the references therein.

{\begin{example}\label{ex: weak convexity}
   \sloppy The $\fonc{f}{\R}{(-\infty,+\infty)}$  defined as $f(x): = |(x-a)(x-b)|$, where ${(a,b)\in\{(a,b)\in \R^2\,|\,a<b\}}$, is $\rho$-weakly convex with $\rho=2.$ 
   %by 
 %virtue of 
 %\Cref{prop:cannarsa}.
 %. As a matter of fact, 
 For every $x\in \R$, it is easy to verify that ${f(x)+ x^2 = \max\{2x^2 -(a+b)x + ab, (a+b)x - ab\}}$
% \begin{equation}
% \begin{aligned}
%      f(x)+ x^2 &= \begin{cases}
%         x^2 -(a+b)x + ab + x^2 &\text{\,\emph{if}\,\,} x\leq a \text{\,\emph{and}\,\,} x \geq b\\
%         -x^2 +(a+b)x - ab  + x^2 &\text{\,\emph{if}\,\,} a\leq x\leq b\\
%     \end{cases}\\
%        & = \max\{2x^2 -(a+b)x + ab, (a+b)x - ab\},
% \end{aligned}
% \end{equation}
%\emph{i.e.\ }function $x\mapsto  f(x)+ x^2$ 
is the point-wise maximum of convex functions; hence, it is convex itself. 
\end{example}

Function $f$ from \Cref{ex: weak convexity}, %is an example of a non-smooth weakly convex function. 
despite its simplicity,  
can be used 
%as a modelling function 
in various applications, {\em e.g.} in
 variational problems as 
a penalty function to promote integer solutions, 
 see \Cref{sec:feasibility}.

{
Below, we summarize some basic  facts 
%that define a characterisation or sufficient conditions for the 
on weak convexity of Fr\'echet differentiable functions.
\begin{proposition} 
\label{prop:subdiff}
Let $\X$ be a Hilbert space and let $h:\X\rightarrow(-\infty,+\infty]$ be  Fr\'echet differentiable on an open and convex set $U\subset\X$. Then, the following holds:

\begin{enumerate}[label=(\alph*)]
\item The following facts are equivalent:
\begin{enumerate}[label=(a.\roman*)]
\item   $h$ is $\rho$-weakly convex on $U$;
\item  for every $x\in U$
\begin{equation} 
\label{eq:descent_ineq}
(\forall\ y\in \X) \quad h(y)\ge h(x)+\langle \nabla h(x)\ |\ y-x\rangle-\frac{\rho}{2}\|x-y\|^{2}.
\end{equation}
\end{enumerate}
\item Let $\nabla h$ be $L_h$-Lipschitz continuous on $U$. Let $x$ and $y$ be in $U$. The following holds.
\begin{enumerate}[label=(b.\roman*)]
\item $|h(x) - h(y) - \langle x-y|\nabla h(y)\rangle |\leq \frac{L_h}{2}\|x-y\|^2$.
    \item $|\langle x-y|\nabla h(x)-\nabla h(y)\rangle |\leq {L_h}\|x-y\|^2$.
\end{enumerate}
\item Let $\nabla h$ be $L_h$-Lipschitz continuous on $\X$. Then $h$ is $L_{h}$-weakly convex on $\X$.
\end{enumerate}
\end{proposition}
\begin{proof}
\textit{(b)} is a well-known result that is generally referred to as \emph{Descent Lemma} (see {\cite[Lemma 2.64]{Bauschke2017}}). We provide the proofs for \textit{(a)} and \textit{(c)}.
    \begin{enumerate}[label={(\alph*)}]
        \item[\textit{(a)}] By {\cite[Proposition 1.1.3]{cannarsa2004semiconcave}} and {\cite[Proposition 17.7]{Bauschke2017}}, for all $x\in U$
\begin{equation*} 
	(a.i) \text{ holds}\begin{array}[t]{l} 
	\Leftrightarrow \ (\forall\ y\in \X) \quad h(y)+\frac{\rho}{2}\|y\|^{2}\ge h(x)+\frac{\rho}{2}\|x\|^{2}+\langle \nabla h(x)+\rho x\ |\ y-x\rangle\\
			\Leftrightarrow \ (\forall\ y\in \X)\quad h(y)\ge h(x)+\langle \nabla h(x)\ |\ y-x\rangle-\frac{\rho}{2}\|x\|^{2}+\langle \rho x\ |\ y\rangle-\frac{\rho}{2}\|y\|^{2}\\
			\Leftrightarrow \ (\forall\ y\in \X)\quad h(y)\ge h(x)+\langle \nabla h(x)\ |\ y-x\rangle-\frac{\rho}{2}\|y-x\|^{2}\\
				\Leftrightarrow\ (a.ii) \text{  holds}.
	\end{array}
\end{equation*}

\item[\textit{(c)}] By \textit{(b)}, for all $x,y\in\X$ 
    $$
    -\frac{L_{h}}{2}\|x-y\|^{2}\le h(y)-h(x)-\langle y-x\ |\ \nabla h(x)\rangle,
    $$
    which, by \textit{(a)}, is equivalent to  $L_{h}$-weak convexity of $h$ on $\X$.
    \end{enumerate}
\end{proof}
We highlight that fact \textit{(c)} in the above proposition proves that in problem \eqref{eq: ourproblem} we are minimizing the sum of two weakly convex functions.
}
Our analysis is based on the concept of \textit{global proximal {$\varepsilon$-subdifferential}}.
% \sloppy \ewa{Shift to the notation on the top of the section?}
% {\color{gray} For a function ${\fonc{h}{\X}{(-\infty,+\infty]}}$, the effective domain, $\dom h$,  is defined as ${\dom h := \{x\in \X\,|\, h(x) < +\infty\}}$ and  we  say that $h$ is \textit{proper} whenever $\dom h \neq \emptyset $. For a set set-valued mapping $M: \X \rightrightarrows \X$, 
% %we  use the notation $\dom M$ to indicate the set
% % \begin{equation*}
%    ${ \dom M := \{x\in \X\,|\, M(x) \neq \emptyset\}.}$
% % \end{equation*}

% }

\begin{definition}[Global proximal ${\varepsilon\text{-subdifferential}}$ \cite{clarke2008nonsmooth,RockWets1998Variational}]\label{def:eps:prox_subdiff}
 Let $\varepsilon\ge0$. Let $\mathcal{X}$ be a Hilbert space. The global proximal ${\varepsilon\text{-subdifferential}}$ of a function $\fonc{h}{\X}{(-\infty,+\infty]}$ at $x_{0}\in\dom\,h$ for $C \geq 0$ is defined as 
 \begin{equation}
 \label{eq:def_prox_subdiff_eps}
     \partial^{\,\varepsilon}_{C} h (x_0) = \left\{v\in\X\,|\,  h(x)-h(x_{0})\ge\langle v\ |\  x-x_{0}\rangle-C\|x-x_{0}\|^{2}-\varepsilon\ \ \forall x\in\X\right\}.
 \end{equation}
 For $\varepsilon=0$, we have
 \begin{equation}
\label{eq:def_prox_subdiff_0}
    \partial_{C} h(x_0) = \{v\in\mathcal{X} \,|\, h(x)\geq h(x_0) + \langle v\ |\ x-x_0\rangle  - C\|x-x_0\|^{2},\;\forall x\in \mathcal{X}\}.
\end{equation}

%In view of \eqref{eq:def_prox_subdiff_0}, $\partial_{0} h(x)$ denotes the subdifferential in the sense of convex analysis. 
The elements of $\partial^{\,\varepsilon}_{C} h(x)$ are called proximal $\varepsilon$-subgradients. %The notation $\partial^{\, \varepsilon} h(x)$ is used when the constant $C$ in \eqref{eq:def_prox_subdiff_eps} is inessential, \emph{\emph{i.e.\ }\ }$v\in\partial^{\, \varepsilon} h(x)$ means that there exists $C\geq 0$ such that $v\in\partial_{C}^{\, \varepsilon} h(x)$.\ewa{are we using this latter notation?} \lele{does not seem so}
\end{definition}

{In analogy to the convex case, for any $\rho\ge 0$, we  use the notation $\Gamma_\rho (\X)$ to indicate the class of proper lower-semicontinuous $\rho$-weakly convex functions defined on $\X$ with values in $(-\infty, + \infty]$. In particular, $\Gamma_0 (\X)$ denotes  the class of proper lower-semicontinuous convex functions defined on $\X$ with values in $(-\infty, + \infty]$.} {By {\cite[Proposition 3.1]{jourani1996Subdifferentiability}}, for a function $h\in \Gamma_\rho(\X)$, the global proximal subdifferential $\partial_{{\rho}/{2}} h(x_0)$ defined by \eqref{eq:def_prox_subdiff_0} coincides with the set of local proximal subgradients which satisfy the proximal subgradient inequality
\eqref{eq:def_prox_subdiff_0}  locally in a neighbourhood of $x_{0}$. 
%Moreover, when $h$ is $\rho$-weakly convex, by {\cite[Theorem 3.1]{jourani1996Subdifferentiability}}, the proximal subdifferential $\partial_{\rho/2}h(x_{0})$ coincides with the Clarke subdifferential $\partial_{c}h(x_{0})$ (see also \cite{Atenas2023Unified}). 
An extensive study of local proximal subdifferentials in Hilbert spaces can be found in 
%Chapter 2 of the book by Clarke, Stern, Ledyaiev, Wolenski - non-smooth Analysis and Control Theory 
{\cite[Chapter 2]{clarke2008nonsmooth}} and in \cite{ Bernard2005_ProxRegular}.
}
% For any set set-valued mapping $M: \X \rightrightarrows \X$, we will use the notation $\dom M$ to indicate the set
% \begin{equation}
%     \dom M := \{x\in \X\,|\, M(x) \neq \emptyset\}.
% \end{equation}
% while for a function $\fonc{f}{\X}{(-\infty,+\infty]}$, the notation $\dom f$ will indicate the set
% \begin{equation}
%     \dom f := \{x\in \X\,|\, f(x) < +\infty\}
% \end{equation}
% \ewa{the domain of $f$  should be defined before subdifferential}
% \begin{proposition}[{ \cite[Proposition 2]{bednarczuk2022calculus}}]
% \sloppy Let $\X$ be a Hilbert space and $h\in\Gamma_{\rho}(\X)$ with $\rho\geq 0$. Then 
%    ${\dom \partial_{\rho/2} h \subset \dom h}$
% and for every $\varepsilon> 0$ {we have the equality ${\dom \partial_{\rho/2}^{\,\varepsilon} h = \dom h}$.}
% \end{proposition}
By \cite[Proposition 2]{bednarczuk2022calculus}, for $h\in\Gamma_{\rho}(\X)$ with $\rho\geq 0$, we have ${\dom \partial_{\rho/2} h \subset \dom h}$
and for every $\varepsilon> 0$ {the equality ${\dom \partial_{\rho/2}^{\,\varepsilon} h = \dom h}$ holds.}
When $\varepsilon=0$, the global proximal subdifferential for functions in $\Gamma_\rho(\X)$ can be seen as a natural generalisation of the Moreau's subdifferential for functions in $\Gamma_0(\X)$: more specifically, for every $x\in \dom \partial _{\rho/2}h$ the existence of a global affine tangent minorant function to $h$ at $x$ is replaced by the existence of a global quadratic concave tangent minorant function.

\section{Proximal Operators}\label{sec:inexact}

We take into account a possibly inexact proximal step for the weakly convex function in algorithm  \eqref{eq: FB upate : 2}. The exact proximal step is considered in {\cite[Section 5]{bohm2021variable}} to infer a complexity bound for the exact forward-backward algorithm. {We use the concept of $\varepsilon$-proximal operator introduced in \cite{bednarczuk2022calculus}, see also \Cref{def:eps_prox} below. Its properties, expressed in terms of global proximal ${\varepsilon\text{-subdifferential}}$s, are investigated in {\cite[Proposition 4, Proposition 5]{bednarczuk2022calculus}}, where a comparison with other concepts of $\varepsilon$-proximal operator in the convex case is provided. The concept of an inexact proximal operator and its computational tractability lie at the core of many proximal methods referring to fully convex case, }where both $f$ and $g$ are convex \cite{barre2022principled,Millan2019InexactProximal,  salzo2012inexact}.
{Proximal operators for nonconvex functions are investigated, e.g. in \cite{hare2009computing}, where an implementable algorithm is proposed for lower-$C^{2}$ functions that are defined as the pointwise maximum of a finite collection of
quadratic function.
}
\begin{definition}[Proximal Map]\label{def:prox} Let $\X$ be a Hilbert space and $\fonc{h}{\X}{(-\infty,+\infty]}$ be a proper $\rho$-weakly convex function on $\X$. For any $\alpha>0$, the set-valued mapping
$\prox_{\alpha h}:  \X  \rightrightarrows  \X $
defined as
\begin{equation}\label{eq:proximal_map}
 \left(\forall y \in \X \right) \quad   \prox_{\alpha h}(y) := \underset{x\in \X}{\argmin}\left\{h(x) + \frac{1}{2\alpha}\|x-y\|^2\right\}
\end{equation}
is called \textit{proximal map} of $h$ at $x$ with respect to parameter $\alpha$.
In general, $\dom \prox_{\alpha h}$ could be empty.
When $1/\alpha> \rho$, or when $h$ is bounded from below,   $\dom \prox_{\alpha h}:= \{x\in \X\,|\, \prox_{\alpha h}(x) \neq \emptyset\}=\X$. And, if $1/\alpha> \rho$,  $\prox_{\alpha h}$ is  single-valued.
\end{definition}

% \subsection{Inexact Proximal Operators}
%In practical contexts, 
{The proximal point mapping is locally Lipschitz continuous
and its set of fixed points coincides with the critical points of the original function.}
It is not always possible to rely on a closed-form expression for the proximity operator of a function (we refer to the web page \cite{proxpage} for a list of examples where these explicit forms are available). 
  %A simple example is represented by the proximal operator of the sum of two or multiple functions. 
%In \cite{Pustelnik2017_ProxOfASum} and in the references therein, the authors provide sufficient conditions for the equality $\prox_{h_1+h_2} = \prox_{h_1}\circ \prox_{h_2}$ to hold in the convex case. 
In general, the computation of the proximal map is an independent optimisation problem.
 It follows that the proximal map might be known up to a certain accuracy only, and %in the framework of proximal splitting methods, 
 %it is important to carry out 
 in our convergence analysis, we take this fact into account. 
 
 In order to deal with the inexact case in \eqref{eq: FB upate : 2} we recall the concepts of $\varepsilon$-solution for an optimisation problem and $\varepsilon$-proximal point.

\begin{definition}[$\varepsilon$-solution]\label{def:epssol} Let $\X$ be a Hilbert space and $\fonc{h}{\X}{(-\infty,+\infty]}$ be a proper function  bounded from below,  $\varepsilon\geq0$ The element $x_\varepsilon$ is  an $\varepsilon$-\textit{solution} to the minimisation problem
 \begin{equation}
     \minimize{x\in \X}{h(x)}
 \end{equation}
 if $h(x_\varepsilon) \leq h(x) + \varepsilon$ for all $x \in \X.$
\end{definition}

%{The proximal point mapping is locally Lipschitz continuous
%and its set of fixed points coincide with the critical points of the original function.}

\begin{definition}[$\varepsilon$-proximal point]\label{def:eps_prox} Let $\X$ be a Hilbert space, $h:\X\rightarrow(-\infty,+\infty]$ be proper and bounded from below, $y\in \X$ and $\alpha>0$, $\varepsilon \geq 0$. Any $\varepsilon$-solution to the proximal minimisation problem 
\begin{equation} 
\label{eq:eps_prox}
\minimize{x\in\X}{ h(x)+\frac{1}{2\alpha}\|x-y\|^{2}}
\end{equation}
 is  a $\varepsilon$-proximal point for $h$ at $y$ with respect to $\alpha$.  The set of all $\varepsilon$-proximal points of $h$ at $y$ is denoted as
\begin{equation}
    \varepsilon\text{-}\prox_{\alpha h}(y) := \left\{x\in \X\,|\, x\  \text{is\;a\;}\varepsilon\text{-solution\;of\;\eqref{eq:eps_prox}} \right\}.
\end{equation}
\end{definition}
When  $h$ is convex, the above definition of an inexact proximal point coincides with the definition in {\cite[Equation (2.15)]{villa2013accelerated}} and the convergence analysis of an inexact forward-backward algorithm  {\cite[Equation (2.15)]{villa2013accelerated}} is proposed in {\cite[Appendix A]{villa2013accelerated}}.\\

%To proceed with our analysis, we consider the following notion of $\varepsilon$-critical point:

\begin{definition}[$\varepsilon$-critical point]\label{def:crit} Let $\X$ be a Hilbert space, ${\fonc{h}{\X}{ (-\infty,+\infty]}}$  be a proper function and $\varepsilon\ge0$, $C\geq 0$.
 A point  $x\in\dom \partial _C^{\,\varepsilon} h $ is  a $C\text{-}\varepsilon$-\textit{critical point} for $h$ if $0\in\partial^{\,\varepsilon}_C h(x)$. The set of $C\text{-}\varepsilon$-critical points is identified as 
\begin{equation}
    \varepsilon \text{-} \operatorname{crit}_C \, h := \{x\in\X\,|\,0\in \partial^{\,\varepsilon}_C h (x) \}.
\end{equation}
\end{definition}

For a function $h\in\Gamma_{\rho}(\X)$, we  consider $C=\rho/2$. 
In order to relate the inexact proximal points of $h\in \Gamma_\rho(\X)$ with $\partial_{\rho/2}^{\, \varepsilon} h$, we use the following lemma.

\begin{lemma}[{\cite[Corollary 1]{bednarczuk2022calculus}}]
\label{cor:proxsub:2}	Let $\X$ be a Hilbert space.
Let $h\in\Gamma_\rho(\X)$ and  $\varepsilon\ge 0$, $\alpha>0$.
If $x_{\varepsilon} \in \varepsilon\operatorname{-prox}_{\alpha h}(y)$, then for any  $e \in \X $ with $\frac{\|e\|^2}{2\alpha}\leq \varepsilon$, we have
\begin{equation}\label{eq:eps:prox_inclusion:2}
\frac{y - x_\varepsilon - e}{\alpha } \in \partial_{ \rho/2}^{\,\varepsilon} h(x_\varepsilon).
\end{equation}
\end{lemma}
  
 \Cref{cor:proxsub:2}  is used in \Cref{sec:inexact FB convergence}, where we investigate the convergence of the inexact scheme.

\begin{remark}
 By   \Cref{def:crit} with $\varepsilon=0$ and {\cite[Theorem 2]{bednarczuk2022calculus}}, for any $x_0\in \prox_{\alpha h}(y) $, we obtain 
\begin{equation}
    0\in \partial_{\rho/2}(h + \frac{1}{2\alpha}\|\cdot - y\|^2)(x_0) = \partial_{\rho/2} h(x_0) + \frac{x_0-y}{\alpha},
\end{equation}
hence {${(y-x_0)}/{\alpha}\in \partial_{\rho/2} h(x_0).$}
\end{remark}

\section{Sharpness and criticality}\label{sec:sharpness}

To achieve convergence for the forward-backward algorithm 
\eqref{eq: FB upate : 2} applied to a nonconvex problem \eqref{eq: ourproblem}, additional assumptions are needed. For example, in  \cite{attouch2013convergence, themelis2018forward, li2015global}, it is assumed that  $f+g$ satisfies the KL inequality with respect to the set of critical points, while in \cite{Atenas2023Unified}  a subdifferential error bound condition is assumed for critical points. 
%On the other hand, we assume sharpness for %the set of global minima.

In this section, we discuss the concept of sharpness, which is one of the main ingredients of the convergence analysis in Section \ref{sec:convergence:analysis}.
{
%Let us start with the analysis of the relationships between the solution set $S$ and $\varepsilon$-critical points of 
Consider the  problem
\begin{equation}
\minimize{x\in \X}{} h(x)
\end{equation}
where the objective function $h$ satisfies the following sharpness condition.}

\begin{definition}[Sharpness, {\cite{Polyak1979,burke1993WeakSharpMinima}}] 
\label{def:sharpness}
Let $\X$ be a Hilbert space and $\fonc{h}{\X}{(-\infty,+\infty]}$ be a function with a nonempty solution set, ${S:=\argmin_{x\in\X}h(x)}\neq\emptyset$. Function $h$  satisfies the sharpness condition locally if there 
exist $\mu >0$ and $\delta>0$ such that
\begin{equation}
    \label{eq:sharpness:local:def}
    \left(\forall x\in B(S,\delta)\right) \quad h(x) - \inf_{x\in \X} h(x) \geq \mu \dist (x,S).
\end{equation}

If \eqref{eq:sharpness:local:def} is satisfied for every $\delta>0$, then the sharpness condition  holds globally and it reads as
\begin{equation}
    \label{eq:sharpness:def}
   %\tag{Sh_{glob}}
    % \begin{split}
        \left(\forall x\in \X\right)\qquad h(x) - \inf_{x\in\X} h(x) \geq \mu \dist (x,S).
\end{equation}

\end{definition}

The sharpness defined above appears in the literature under different names, {\em e.g.} weak sharp solutions, error bound and was investigated on different levels of generality. For the detailed discussion see {\em e.g.} \cite{cuong2022error} and the references therein. The following elementary example shows that weakly convex globally sharp functions exist.
%(see also \Cref{sec:feasibility} devoted to the feasibility problem).

\begin{example}\label{ex:sharpness}
\sloppy The function $f$ from \Cref{ex: weak convexity} is globally sharp with solution set $S=\{a,b\}$. Take ${\delta = \left|\frac{a-b}{2} \right|}$ so that ${\overline{B}(a,\delta)}\cap {\overline{B}(b,\delta)} = \{\frac{a+b}{2}\}$. If $x = \frac{a+b}{2}$, 
then
$
   {f(x) =  |-\left(\frac{a-b}{2}\right)\left(\frac{a-b}{2}\right) | = |\frac{a-b}{2} | \dist(x,S).}
$
If $x\in B(a,\delta)$, 
%so $x\notin B(b,\delta)$, and 
then ${f(x) \geq \delta |x-a| = \delta \dist (x,S)}$. If $x\in B(b,\delta)$, then 
%$x\notin B(a,\delta)$, and 
${ f(x) \geq \delta |x-b| = \delta\dist (x,S)}$. If $x\notin B(a,\delta)$ and $x\notin B(b,\delta)$, then $f(x) \geq  \delta |x-a|$ and $f(x) \geq \delta |x-b|$.
Hence, for every $x\in\R$m
  ${f(x) \geq \delta \min\{ |x-a|, |x-b|\} =  \delta \dist (x,S)}$, $f$ is {globally sharp} with constant  { $\delta>0$} for  $\delta \leq\left |\frac{a-b}{2}\right |$.
\end{example}

{
The concept of sharpness allows us to localize a neighbourhood around the solution set $S$, where there are no other $\varepsilon$-critical points than solutions. }

% Below, we investigate the position of the $\varepsilon$-critical points of $h$ which are not in $S$.
{ 
\begin{proposition}[Position of solutions relative to $C$-$\varepsilon$-critical points]
\label{prop:eps:statpoint}
Let $\X$ be a Hilbert space and ${\fonc{h}{\X}{(-\infty,+\infty]}}$ be a proper function on $\X$.  
Let 
$S = \argmin_{x\in\X} h(x)$ and
$h$ satisfy the sharpness condition \eqref{eq:sharpness:def} with a constant $\mu>0$. {
For any $0\leq \varepsilon< \mu^2/4C$ and any $x \in\varepsilon\operatorname{-crit}_C\;(h)\setminus S $, $C\geq 0$, 
\begin{equation} 
\label{eq:rings}
\text{either \ } \dist(x,S) \geq \tau_2  \text{  or  } 0< \dist(x,S) \leq \tau_1(\varepsilon),
\end{equation}
where \begin{equation}
\label{eq: t12}
    \tau_{1}(\varepsilon): = \frac{\mu - \sqrt{\mu^2 - 4 C\varepsilon}} {2C} \qquad  \text{and}\qquad\tau_{2}(\varepsilon): = \frac{\mu + \sqrt{\mu^2 - 4 C\varepsilon}} {2C}.
\end{equation}
}
\end{proposition}
\begin{proof}
Let $x \in\varepsilon\operatorname{-crit}_C\;(h)\setminus S $. By \Cref{def:crit}, $0\in \partial_{C}^{\varepsilon} h(x)$,
which means that
\begin{equation}
\left(\forall y \in \X\right) \qquad \qquad h(x) - h(y) \leq C\|x-y\|^2 + \varepsilon.
\end{equation}
%In particular, when $y=\overline{x}\in S$, 
By the global sharpness of $h$, 
\begin{equation}
(\forall \overline{x}\in S) \qquad \mu \dist(x,S) \leq C\|x-\overline{x}\|^2 + \varepsilon.
\end{equation}
By taking the infimum in the right-hand side over all $\overline{x}\in S$ we obtain the inequality
\begin{equation}
    \label{eq:eps:quadratic_equation}
    \mu \operatorname{dist}(x,S) \leq C\operatorname{dist}^2(x,S) + \varepsilon.
\end{equation}
which is  quadratic  with respect to $\dist(x,S)$
%we obtain from \eqref{eq:eps:quadratic_equation} that
and  either $ \dist(x,S) \geq \tau_2(\varepsilon)  $ or $0< \dist(x,S) \leq \tau_1(\varepsilon)$.
\end{proof}
}

\begin{remark} 
\label{rem:local_sharp}
When the global sharpness \eqref{eq:sharpness:def} is replaced in \Cref{prop:eps:statpoint} by the local sharpness \eqref{eq:sharpness:def}, then \eqref{eq:rings} remains unchanged provided $\delta>\tau_{2}(\varepsilon)$.
\end{remark}

In view of \Cref{prop:eps:statpoint}, given $0\le\varepsilon<\mu^2/C$, %the solution set $S$ and 
for any $C$-$\varepsilon$-critical point $x\in \X$ of $h$ which does not belong to $S$, either the point $x$ is  in a neighbourhood of $S$ of radius $\tau_1(\varepsilon)$ or its distance from $S$ is bounded from below by a positive quantity $\tau_2(\varepsilon)$. 

\begin{remark} 
\label{rmk:statpoint}
When $\varepsilon=0$, for a ${C}$-critical point $x$ that does not belong to $S$, by \cref{prop:eps:statpoint} we obtain
% \begin{equation}\label{eq:DavisInequality}
%     \frac{\mu}{C} \leq \dist(x,S)
% \end{equation}
% or, equivalently,
\begin{equation}\label{eq:DavisInequality1}
    \operatorname{crit}_C (h) \setminus S\subset \left\{ x\in \X\ |\ \frac{\mu}{C} \leq \dist(x,S)\right\},
\end{equation}
% % \end{itemize}
which coincides with {\cite[ Lemma 3.1]{Davis2018SubgradientMethodsSharpWeaklyConvex}}, that is stated for $\rho$-weakly convex functions and $C=\rho / 2$.
\end{remark}

\Cref{prop:eps:statpoint}, in the form considered in \Cref{rmk:statpoint}, is illustrated in the following example, 
% where inequality \eqref{eq:DavisInequality} is attained.

{
\begin{example}\label{ex:critical_points}
Let us consider function $f$ from \Cref{ex: weak convexity}. We notice that $x = \frac{a+b}{2}$ is a $(\rho/2)$-$\varepsilon$-critical point in the sense of \Cref{def:crit} for $\rho = 2$ and $\varepsilon=0$. %We start
%by rewriting $ f$ in the following manner: 
Observe that, for every $x\in\R$,
\begin{equation*}
\begin{aligned}
  f(x) &= \max \{(x-a)(x-b), -(x-a)(x-b)\} = \max \{x^2 - (a+b)x + ab, -x^2-ab + (a+b)x\}
\end{aligned}\end{equation*}
which implies that, for every $x\in\R$, $ f(x) \geq  -x^2-ab + (a+b)x$.
We now notice that, for any $(a,b)\in\R^2$, the identity $ {(a-b)^2}/{4} -  {(a+b)^2}/{4} = -ab $ holds. It follows that
% \begin{equation*}
%     \frac{(a-b)^2}{4} -   \frac{(a+b)^2}{4} = \left( \frac{a-b}{2} + \frac{a+b}{2}\right) \left(\frac{a-b}{2} - \frac{a+b}{2}\right) = -ab
% \end{equation*} 
\begin{equation*}
  (\forall x \in \R) \qquad f(x) \geq  -x^2 + (a+b)x + \frac{(a-b)^2}{4} -   \frac{(a+b)^2}{4}  = -\left( x - \frac{a+b}{2}\right)^2 + \frac{(a-b)^2}{4}.
\end{equation*}
Since $ f\left( \frac{a+b}{2}\right) = \frac{(a-b)^2}{4}$, we get $f(x) -  f\left( \frac{a+b}{2}\right)  \geq    -\left( x - \frac{a+b}{2}\right)^2$ for $x\in\R$, 
which, according to \Cref{def:eps_prox}, implies that $0 \in \partial_{\rho/2}^{\,\varepsilon}  f \left(\frac{a+b}{2}\right)$ for every $\varepsilon\geq 0 $ and  $\rho = 2$. In addition, considering the fact that function $ f$ satisfies the global sharpness condition with constant $\mu = \left| \frac{a-b}{2} \right|$ (as shown in \Cref{ex:sharpness}), we have 
% inequality \eqref{eq:DavisInequality} is satisfied at $\overline{x} = \frac{a+b}{2}$:
\begin{equation*}
    \frac{2\mu}{\rho} = \frac{2 \left| \frac{a-b}{2} \right|}{2} = \left| \frac{a-b}{2} \right|  = \dist\left( \frac{a+b}{2},S\right).
\end{equation*}
\end{example}

For a weakly convex function $h\in \Gamma_\rho (\X)$, 
and the set $[h=0]=\{x\in\X\,|\, h(x) =0\}$, we  define the  Kurdyka-{\L}ojasiewicz inequality as in \cite{bolte2010characterizations}. For a given ${r_0}>0$, the set $\mathcal{K}(0,{r_0})$ represents the class of \emph{desingularisation functions}, which is defined as
\begin{equation}
    \mathcal{K}(0,{r_0}) = \{\varphi \in \mathcal{C}[0,{r_0})\cap\mathcal{C}^1(0,{r_0}) \,|\, \varphi(0) = 0\text{\;and\;}\phi'(r)>0 \;\forall r \in (0,{r_0})\}.
\end{equation}
\begin{definition}[Kurdyka-{\L}ojasiewicz Inequality]\label{def:KL:2}
Let $\delta_0, r_0>0$, {$x^* \in [h=0]$} and $\mathcal{K}(0,r_0)$ be a class of desingularisation functions. We say that the Kurdyka-{\L}ojasiewicz inequality holds at $x^*$ if there exist $r\in (0,r_0)$ and a function $\varphi\in \mathcal{K}(0,r_0)$ such that 
\begin{equation}
    \label{eq:KL:loc}
    %\tag{$K{\text{\L}}_{\text{loc}}$}
   \inf\{\|z\| \ |\ z\in \partial_{\rho/2}(\varphi\circ h)(x)\} \ge 1,\ \ \forall\,x\in \overline{B}(x^*, \delta_0)\cap [ 0 < h \le r],
\end{equation}
where $[ 0 < h \le r] = \{x\in\X\,| 0< h(x) \leq r\} $.
\end{definition}

%{F the rest of this section,   $h$ is a non-negative function with non-empty set of minimisers $S$.} 
We close this section by addressing the question of the relationship between \eqref{eq:KL:loc} and \eqref{eq:sharpness:local:def}. 
%for   a non-negative function $h$ with a non-empty set of minimisers $S$.} The following lemma holds.

\begin{proposition} 
\label{cor:sh_versus_kl} 
Let $\X$ be a Hilbert space and $h\in \Gamma_{\rho}(\X)$, $h\geq 0$. Let $x^{*}\in S = \argmin_{x\in\X} h(x) = [h\leq 0]\neq\emptyset$.
Consider the following facts:
\begin{description} 
 \item \textit{(i)} there exist $r_1>0$, $\delta_1>0$, $\mu>0$ such that for every $x\in  B(x^*, \delta_1)\cap [ 0 < h(x) < r_1]$
 \begin{equation}\label{eq:prop:sh_loc}
     % \quad \left(\forall x\in \mathcal{B}(S,\delta)\right)&&
     h(x) - \inf_{x\in \X} h(x) \geq \mu \dist (x,S);
 \end{equation}
 % \eqref{eq:sharpness:local:def} 
 \item \textit{(ii)} \sloppy there exist $\overline r>0$, $r_2\in(0,\overline{r})$, $\delta_2>0$ and $\varphi\in\mathcal{K}(0,\overline{r})$ such that %\eqref{eq:KL:loc} holds at $x^{*}$
 for every
${x\in \overline{B}(x^*, \delta_2)\cap [0 < h(x) \le r_2]}$
 \begin{equation}\label{eq:prop:kl_loc}
     \inf\{\|z\| \ |\ z\in \partial_{\rho/2}(\varphi\circ h)(x)\} \ge 1.
 \end{equation}
\end{description}
Then $(i)\implies (ii)$. If $h$ is continuous and $\varphi = \frac{1}{\mu}\operatorname{Id} $ for some $\mu>0$ (where  $\operatorname{Id}$ is the identity operator over $\X$), then $(ii) \implies (i)$.
 \end{proposition}

}

\begin{proof}
    See Appendix \ref{app:proof_of_equivalence}.
\end{proof}
While sharpness can imply the KL property, the other implication is not always true. Consider the function $h(x)=x^2$: clearly, this function satisfies KL inequality, but it is not sharp. In fact, it is sharp with order $2$, or we can say that $h$ satisfies quadratic growth condition \cite{drusvyatskiy2018error}.
As observed in \cite{Atenas2023Unified}, the quadratic growth condition is related to the subdifferential error bound around the set of critical points.

%\ewa{in [1] the definition is local, but I do not see $|varepsilon$, please check!}

\begin{definition}[ c.f. {Definition 3.1,\cite{Atenas2023Unified}}] \sloppy We say that the subdifferential error bound holds for a problem of the form $\minimize{x\in\X} h(x)$ where ${\fonc{h}{\X}{(-\infty,+\infty]}}$ is bounded from below, if, for every ${v\geq \inf_{x\in\X} h(x)}$, there exists $\ell >0$ such that whenever $x\in\X$, $h(x)\leq v$, the following is true:
\begin{equation}
\dist(x,\operatorname{crit} h) \leq \ell \dist(0, \partial_P h (x)),
\end{equation}
where $\operatorname{crit} h = (\partial_P h)^{-1}(0)$ and  $$\partial_P h(x_0) = \{v\in \X \mid \exists C\geq 0 \text{ s.t. } h(x)-h(x_0) \geq \langle v \mid x-x_0\rangle - C \|x-x_0\|^2,\, \forall x\in\X\}. $$
\end{definition}
In fact, for weakly convex functions, it has been proved that the subdifferential error bound is a sufficient condition for the quadratic growth condition \cite{cuong2022error,liao2024error}. In the example below, we show that global sharpness does not imply that the subdifferential error is global.

\begin{example}
{Consider the function defined for every $x\in\R$  as $\bar{f}(x)=|x^{2}-1|$, which is an instance of the function from \cref{ex:critical_points} with $a=-1$ and $b=1$: this does not satisfy the  global subdifferential error bound but only the local version, i.e. there exists $\varepsilon>0$ such that
$$
(\forall x\in B(S,\varepsilon))\quad \dist(x,S)\le\kappa \dist(0,\partial_{P} \bar{f}(x)).
$$
Let us consider
$$
g(x)=\bar{f}(x)+x^{2}=\left\{\begin{array}{ll}
x^{2}-1+x^{2}&x^{2}-1>0 ,\\
-x^{2}+1+x^{2}& x^{2}-1\le 0\end{array}\right. .
$$
And $\dist(0,\partial_{P} \bar{f}(x))=\dist(2x, \partial g(x))$, where
$$
\partial g(x)=\left\{\begin{array}{ll}
4x&x^{2}-1>0 ,\\
0& x^{2}-1\le 0\end{array}\right. .
$$
As a consequence, 
$$
\dist(0,\partial_{P} \bar{f}(x))=d(2x, \partial g(x))=2|x|,
$$
while $\dist(x,S) = \min\{ |x-1|,|x+1|\}$ (c.f. the dashed lines picture of $\dist(x,S)$ in \cref{fig:comparisons}). This means that the subdifferential error bound holds locally but not globally, while the sharpness condition holds globally.
}
\end{example}

{ \Cref{cor:sh_versus_kl} illustrates that, for weakly convex functions, the inequality defining the sharpness condition from \Cref{def:sharpness} is closely related to the KL property as defined in \Cref{def:KL:2}. Note that \eqref{eq:prop:sh_loc} holds around a given $x^* \in S$, while the local sharpness condition requires the same relation \eqref{eq:prop:sh_loc} to hold locally around the solution set $S$. Details about the relation between sharpness and KL in the finite-dimensional setting are discussed in \cite{Bai2022Equivalence}. 

However, as stated in \cite{bonettini2020convergence}, when the algorithms under analysis involve $\varepsilon$-proximal points, in order to reproduce the abstract convergence scheme from \cite{attouch2013convergence}, one would need to use an approximate version of the KL inequality involving (proximal) $\varepsilon$-subdifferentials. As a consequence, this raises the question of under which conditions an objective function satisfies this approximate KL inequality.  Unfortunately, it can be shown that this approximate property fails to hold even for basic convex functions in the finite-dimensional setting, as illustrated in \cite[Example 2]{bonettini2020convergence} for the case of the absolute value. This issue is one of the reasons behind our choice to adopt the assumption of sharpness of the objective function in order to carry out the convergence analysis of \Cref{alg:epsFB} in the next sections. Note that in \Cref{theo:sharpness 1}, we will provide direct proof of the sharpness of the objective function of the feasibility problems we consider in \Cref{sec:feasibility}.}
% {\color{gray}
% \begin{remark}
%     If $h$ satisfies \eqref{eq:KL:loc} at $x^*\in S$, then, by \Cref{cor:sh_versus_kl}, there exist $\eta,r_0 >0$ such that
% \begin{equation}
%     (\forall x\in B(x^*,\eta)\cap [0<h<r_0])\qquad h(x)-h(x^*) \geq c\dist(x,S).
% \end{equation}
% Combining this with \eqref{eq:KL:loc}, we have \hung{ 
% \begin{equation}
%     \dist(0,\partial_{\rho/2} h(x)) \geq \frac{1}{r_0} (h(x)-h(x^*)) \geq \frac{c}{r_0} \dist (x,S),
% \end{equation}
% }
% for all $x\in B(x^*,\eta)\cap [0<h<r_0]$, which is the subdifferential error bound c.f. \cite[Theorem 3.11]{Kruger2019HlderEB}.
% \end{remark} }

\section{The Inexact FB Algorithm} \label{sec:inexact FB convergence}

%Let us give the basic setting for the minimization of problem 
%\eqref{eq: ourproblem}. 
\Cref{assu:problem} summarises the standing assumptions on functions $f$ and $g$, which are made in all the results of the subsequent sections.
\begin{assumption}
    \label{assu:problem} Let the following facts hold:
    \begin{enumerate}[parsep=0cm, itemsep=0cm, topsep=0cm, label=(\roman*)]
    \setlength{\itemindent}{+.3in}
        \item function $f\in\Gamma_\rho (\X)$ is bounded from below;
        \item function $g\in \Gamma_0 (\X)$ is {Fr\'echet} differentiable on $\X$ with a $L_g$-Lipschitz continuous gradient;
        %\item $\dom f \subset \dom g$; %$\dom f \cap \dom g \neq \emptyset$; \lele{ ensures the well definition the iterates}
        \item the set $S = \argmin_{x\in\X}(f+g)(x)$ is non-empty.
    \end{enumerate}
\end{assumption}

We discuss the following forward-backward scheme:\\

\begin{algorithm}[H]\caption{Inexact Forward-Backward}\label{alg:epsFB}
\textbf{Initialise:} $x_0 \in \dom (f+g)$\\
\textbf{Set:} $(\varepsilon_t)_{t\in\N}$ non-negative, $(\alpha_t)_{t\in\N}$ positive; \\
\textbf{For}{ $t=1,2,\dots$ {\normalfont\textbf{until} convergence}}{
\begin{align}
y_{t}\;&=x_{t}-\alpha_{t}\nabla g(x_{t}) \\
x_{t+1}\;&\in{\varepsilon_{t}}\text{-prox}_{\alpha_t f}(y_t) %\text{using an {oracle}}
\end{align}
}
\end{algorithm}
\vspace{0.5cm}

 In the analysis below, we do not tackle the problem of how to compute the $\varepsilon\operatorname{-}$proximal point of  $f$ at any given $y$. Instead, we assume the existence of an oracle which provides us with an $\varepsilon$-proximal point. For instance, see Appendix \ref{app:inexprox}.  %Recall that we defined $\varepsilon$-proximal points as $\varepsilon$-solutions of the corresponding optimisation problem, see \Cref{def:eps_prox}. 

By applying \Cref{cor:proxsub:2}, we investigate the decreasing behaviour of the objective function in \eqref{eq: ourproblem} for a sequence generated by \Cref{alg:epsFB}. 

% \lele{this generalises \cite[Proposition 7]{bayram2015convergence}}
\begin{proposition}
\label{prop:eps:FB estimate g-convex}
Let $\X$ be a Hilbert space and $\fonc{f,g}{\X}{(-\infty,+\infty]}$ satisfy \Cref{assu:problem}.
%Let $f\in \Gamma_\rho(\X)$ and
% $\fonc{g}{\X}{(-\infty,+\infty]}$ be proper, convex and %differentiable on its domain with $L_g$-Lipschitz continuous gradient, ${\dom f \cap \dom g \neq \emptyset}$. 
Let $\left(x_t\right)_{t\in\N}$ be the sequence generated by \Cref{alg:epsFB}. Then, for any $x\in \X$, 
%the following estimate holds:
\begin{equation}\label{eq:eps:FB_estimate}
\begin{aligned}
    (f+g)(x)   - (f+g)&(x_{t+1}) \geq \innerprod{\frac{x_t - x_{t+1}}{\alpha_t}}{x-x_{t+1}} \\
    &  - \frac{\rho}{2}\|x-x_{t+1}\|^2 - \frac{L_g}{2}\|x_{t}- x_{t+1}\|^2    
    - \varepsilon_t - \sqrt{\frac{2\varepsilon_t}{\alpha_t}}\| x - x_{t+1}\|.\\
    \end{aligned}
\end{equation}
Moreover,  when $2/\alpha_t > \rho +L_g$, for all $t\in\mathbb{N}$
\begin{equation}
\label{eq:decrease}
     \left(f+g\right)\left(x_t\right)-\left(f+g\right)\left(x_{t+1}\right) >  - \varepsilon_t - \sqrt{\frac{2\varepsilon_t}{\alpha_t}}\| x_t - x_{t+1}\|.
\end{equation}%}
%when $2/\alpha_t > \rho+1 +L_g$.
\end{proposition}

\begin{proof}
By \Cref{alg:epsFB} and \Cref{cor:proxsub:2}, for any $t\in\mathbb{N}$, there exists $e_t\in \X$,  $\frac{\|e_t\|^2}{2\alpha_t}\leq \varepsilon_t$, such that
\begin{equation}
    \frac{x_t - \alpha_t \nabla g(x_t) - x_{t+1} - e_t}{\alpha_t} \in \partial_{\rho/2}^{\,\varepsilon_t} f(x_{t+1}).
\end{equation}
By definition of proximal ${\varepsilon\text{-subdifferential}}$,  for any $x\in\X$
\begin{align}\label{eq:inequality01}
    f\left(x\right)-f\left(x_{t+1}\right)	&\geq\left\langle \frac{x_{t}-x_{t+1}}{\alpha_{t}}-\nabla g\left(x_{t}\right) | x-x_{t+1}\right\rangle -\frac{\rho}{2}\left\Vert x-x_{t+1}\right\Vert ^{2} - \left\langle \frac{e_t}{\alpha_t} |  x - x_{t+1}\right\rangle -\varepsilon_t.
\end{align}
By the convexity of $g$ and Descent Lemma, we obtain
\begin{equation}
\begin{aligned}\label{eq:inequality00}
    \left\langle -\nabla g\left(x_{t}\right)\ |\ x-x_{t+1}\right\rangle 
    &\geq g\left(x_{t+1}\right) - g(x) 
    -\frac{L_{g}}{2}\left\Vert x_{t}-x_{t+1}\right\Vert ^{2}.
    \end{aligned}
\end{equation}
On the other side,  by applying the Cauchy-Schwarz inequality, for any $x\in\X$
\begin{equation}\label{eq:inequality2}
    \left\langle \frac{e_t}{\alpha_t}\ |\  x - x_{t+1}\right\rangle \leq \frac{1}{\alpha_t}\|e_t\|\| x - x_{t+1}\|% \leq \frac{1}{2\alpha_t}\left( \|e_t\|^2 + \|x - x_{t+1}\|^2\right) \leq \varepsilon_t + \frac{1}{2\alpha_t}\|x- x_{t+1}\|^2.
    \leq \frac{1}{\alpha_t}\sqrt{2\alpha_t\varepsilon_t }\| x - x_{t+1}\| \leq \sqrt{\frac{2\varepsilon_t}{\alpha_t}}\| x - x_{t+1}\|.
\end{equation}
By plugging \eqref{eq:inequality00} and \eqref{eq:inequality2} into \eqref{eq:inequality01} we obtain \eqref{eq:eps:FB_estimate}.% for any $x\in\X$
% \begin{equation}
% \begin{aligned}
% \label{eq:inequality_f_g}
%     &(f+g)(x) - (f+g)(x_{t+1}) \\
%     &\geq \innerprod{\frac{x_t - x_{t+1}}{\alpha_t}}{x-x_{t+1}}  - \frac{\rho}{2}\|x-x_{t+1}\|^2 - \frac{L_g}{2}\|x_{t}- x_{t+1}\|^2 - \varepsilon_t - \sqrt{\frac{2\varepsilon_t}{\alpha_t}}\| x - x_{t+1}\|.\\
%     \end{aligned}
% \end{equation}

For the second part, since \eqref{eq:eps:FB_estimate} holds for any $x\in \X$, substituting $x=x_t$ yields
\begin{align}
    \left(f+g\right)\left(x_t\right)-\left(f+g\right)\left(x_{t+1}\right)&\geq \left(\frac{1}{\alpha_t}-\frac{\rho}{2}-\frac{ L_{g}}{2}\right) \left\Vert x_{t}-x_{t+1}\right\Vert ^{2} - \varepsilon_t - \sqrt{\frac{2\varepsilon_t}{\alpha_t}}\| x_t - x_{t+1}\|.
\end{align}
The assumption $2/\alpha_t > \rho +L_g$ yields
\begin{equation}
     \left(f+g\right)\left(x_t\right)-\left(f+g\right)\left(x_{t+1}\right)>  - \varepsilon_t - \sqrt{\frac{2\varepsilon_t}{\alpha_t}}\| x_t - x_{t+1}\|.
\end{equation}
\end{proof}

\begin{remark}\label{rem:FB estimate g-convex}
When $\varepsilon_t=0$ for every $t\in\N$, inequality \eqref{eq:decrease} from \Cref{prop:eps:FB estimate g-convex} reduces to
\begin{equation}\label{eq:FB estimate g-convex}
     \left(f+g\right)\left(x_t\right)-\left(f+g\right)\left(x_{t+1}\right)> 0,
\end{equation}
 \emph{i.e.\ }the objective values diminish in a strictly monotone way {(even if both $f$ and $g$ are weakly convex) (for finite-dimensional setting c.f. {\cite[Proposition 7]{bayram2015convergence}})}. If 
 %the objective function is lower bounded,
 $\inf( f+g)>-\infty$, then $(f+g)(x_t)$  converges as $t$ goes to infinity. In addition, we observe that the restriction that we have on the step size 
$\alpha_t < \dfrac{2}{L_g + \rho}$
generalises the one of the convex case (\emph{i.e.\ }when $\rho=0$) which is $\alpha_t < 2 / L_g$  (see \cite{combettes2005signal}).
\end{remark}

From now on, we consider fixed  {$\alpha_{t}=\alpha$ and $\varepsilon_{t}=\varepsilon$ for  $t\in\mathbb{N}$}. 
We adopt the following assumption.

\begin{assumption}\label{assu:conditions}
 We choose $\varepsilon \geq 0$ and $\alpha >0$ satisfying the following inequalities:
\begin{align} 
\label{eq:cond:epsilon}&\varepsilon<\frac{\mu^{2}}{2\left(\rho+1\right)}\min\left \{\frac{1}{L_g+1},\frac{1}{\rho+2}\right\},\\ %\\
\label{eq:cond:alpha:1}&\frac{2\varepsilon (\rho+1)}{\mu^2 -2\varepsilon(\rho+1)} \le \alpha < \min\left\{\frac{1}{L_g}, \frac{1}{\rho+1}\right\}.
%\label{eq:cond:alpha:2}\alpha &\geq \frac{2\varepsilon (\rho+1)}{\mu^2 -2\varepsilon(\rho+1)}\\
% \label{eq:cond:alpha:1} &\frac{1}{L_g}\geq \alpha.
\end{align}
\end{assumption}
The following quantities {play an important role in the convergence analysis in \Cref{sec:convergence:analysis}}:
\begin{align}
   % \label{eq:E minus:c}
    E^{-} := \frac{\mu-\sqrt{\mu^{2}-2\varepsilon_{}\left(\rho+1\right)\frac{\alpha_{}+1}{\alpha_{}}}}{\rho+1}\text{\quad and\quad}
   % \label{eq:E plus:c}
    E^{+} := \frac{\mu+\sqrt{\mu^{2}-2\varepsilon_{}\left(\rho+1\right)\frac{\alpha+1}{\alpha_{}}}}{\rho+1}. 
\end{align}

\begin{remark}
\label{rem:tubeE}
%We notice that the following inequalities holds: 
% {\sloppy We have}
%\begin{equation}
   We have $\tau_1(\varepsilon)\leq E^- < E^+ \leq \tau_2(\varepsilon)$,
%\end{equation}
{ $\tau_1(\varepsilon)$, where $\tau_2(\varepsilon)$ are as in %\linebreak
\Cref{prop:eps:statpoint}.}
\end{remark}

\begin{remark}\label{rem:exact_case}

When $\varepsilon=0$, \Cref{assu:conditions} simplifies to 
  \begin{equation*}
 \alpha < \min \left\{ \frac{1}{L_g} , \frac{1}{\rho+1}\right\}. 
\end{equation*}
This implies $2 > \alpha(\rho + L_g)$
which  
%satisfies 
{coincides with } the condition in {\eqref{eq:decrease}} in \Cref{prop:eps:FB estimate g-convex}.
%ensuring that
%the sequence of objective values %$\left((f+g)(x_t)\right)_{t\in\N}$ is monotonically decreasing. 
To conclude, we observe that in the case of exact proximal computations ($\varepsilon=0$, {\emph{c.f} \Cref{rmk:statpoint}} ), 
\begin{equation*}
E^- = \tau_1(0)= 0 \quad \text{and}\quad E^+ = \tau_2(0)  = \frac{2\mu}{\rho}.
\end{equation*}
%{\emph{c.f} \Cref{rmk:statpoint}}.
\end{remark}

\section{Convergence analysis}\label{sec:convergence:analysis}

The aim of this section is to show the results presented in \Cref{theo:Et constant convergence}, which describes the behaviour of the sequence $\text{dist}(x_t,S)$. 
In view of \Cref{rem:exact_case}, when $\varepsilon=0$, condition \eqref{eq:assuA} reduces to $0<\text{dist}(x_t,S)<\frac{2\mu}{\rho}$. 

Condition \eqref{eq:assuA} states that if, at some moment, the iterates generated by the algorithm get sufficiently close to the solution set, we can deduce convergence properties in terms of distance from the solution set.

The following result is crucial in the next sections to provide the (local) convergence of the iterates to a global solution of \eqref{eq: ourproblem}.
The case $\varepsilon=0$ is analyzed in more detail in \Cref{sec:convergence:exact}.

\begin{theorem} \label{theo:Et constant convergence}
 Let $\X$ be a Hilbert space, $\fonc{f,g}{\X}{(-\infty,+\infty]}$ satisfy \Cref{assu:problem} and  $f+g$ satisfy the sharpness condition \eqref{eq:sharpness:def} with constant $\mu>0$. Assume that $\varepsilon\geq0$, $\alpha>0$ are chosen according \Cref{assu:conditions}, 
and $\left(x_t\right)_{t\in\N}$ is  generated by \Cref{alg:epsFB}. 
The following holds true.
\begin{description} 
\item[--] If there exists $t_0 \in \N$ such that 
%the following holds 
\begin{equation}\label{eq:assuA}
E^-<\operatorname{dist}(x_{t_0},S) < E^+,
\end{equation}
then there exists a constant $\zeta >1$ such that
\begin{equation}\label{eq:theo:conv_rate}
    (\forall t\geq t_0)\qquad \dist^2(x_{t},S)-(E^-)^2 \leq \left(\frac{1}{\zeta}\right)^{t-t_0}( \dist^2(x_{t_0},S)-(E^-)^2).
\end{equation}
%implying 
Consequently,
       $\limsup_{t\to\infty} \dist(x_{t},S) \leq E^-.$
\item[--] If 
\begin{equation}\label{eq:assu1} (\forall t\geq t_0)\qquad E^- < \operatorname{dist}(x_{t},S) < E^+,
\end{equation}
then 
%\begin{equation}
    $\lim_{t\to +\infty} \dist(x_t,S) = E^-.$
%\end{equation}
\end{description}
\end{theorem}

The following two lemmas are used in the proof of \Cref{theo:Et constant convergence}.

\begin{lemma}\label{lem:inthetube:const}
 
 %Under the assumptions 
Let the assumptions of \Cref{theo:Et constant convergence} be satisfied. We have the following:
%we have 
\begin{equation}
\label{eq:dist:xtplus}
(\exists t_0\in\N)\quad  \dist\left(x_{t_0},S\right)\leq
E^+ \Rightarrow \quad (\forall t\geq t_0) \quad \dist\left(x_{t},S\right)\leq
E^+.
\end{equation}

\begin{equation}
\label{eq:dist:xtminus}
(\exists t_0 \in\N) \quad \dist\left(x_{t_0},S\right)\leq
E^- \Rightarrow \quad (\forall t\geq t_0) \quad \dist\left(x_{t},S\right)\leq
E^-. 
\end{equation}
\end{lemma}

\begin{proof}
See \Cref{sec:lem:inthetube:const}.
\end{proof}

\begin{lemma}
\label{prop:estimation with Et}
{ 
Under the assumptions of \Cref{theo:Et constant convergence},
if, for some $ t_0\in \N$, we have 
\begin{equation}
% \label{eq:dist:xtplus}
\text{dist}\left(x_{t_0},S\right)\leq
E^+,
\end{equation}
then for every $t\geq t_0$ there exists $\zeta_{t+1}\geq 1$ such that 
\begin{equation}
\label{eq:eps:contraction}
    \zeta_{t+1} \left(\operatorname{dist}^2 (x_{t+1},S) -(E^-)^2\right) \leq  \operatorname{dist}^2 (x_{t},S)-(E^-)^2.
\end{equation}
Consequently, $\dist(x_{t+1},S) \leq \dist(x_t,S)$ provided $\dist(x_{t+1},S) >E^-$.}
\end{lemma}

\begin{proof}
See \Cref{sec:prop:estimation with Et}.
\end{proof}
\Cref{lem:inthetube:const} and \Cref{prop:estimation with Et} hold true when 
%the sharpness condition 
\eqref{eq:sharpness:def} is replaced by \eqref{eq:sharpness:local:def} for a sufficiently large $\delta>0$.
{
\begin{proof}[\textbf{Proof of \Cref{theo:Et constant convergence}}]

By \eqref{eq:eps:contraction},
%in \Cref{prop:estimation with Et}  and from the fact that $E^-$ is constant for every $t$, 
there exists $\zeta_{t+1}\geq 1$ such that
\begin{equation}
 (\forall t\geq t_0)\qquad  \zeta_{t+1} \left(\operatorname{dist}^2 (x_{t+1},S) -(E^-)^2\right) \leq  \operatorname{dist}^2 (x_{t},S)-(E^-)^2.
\end{equation}
% where {$\zeta_{t+1} := 1-\alpha \rho -\alpha+ \frac{2\alpha \mu}{\operatorname{dist}(x_{t+1},S)+ E^-}\geq 1.$}
From this, we infer that, for every $t\geq t_0$, the following holds:
\begin{equation}\label{eq:inequality_prod}
\begin{aligned}
\operatorname{dist}^2 (x_{t+1},S) -(E^-)^2 &\leq  \frac{1}{\zeta_{t+1}}\left(\operatorname{dist}^2 (x_{t},S)-(E^-)^2 \right) \leq  \left(\prod_{s=t_0}^{t}\frac{1}{\zeta_{s+1}}\right)\left(\operatorname{dist}^2 (x_{t_0},S)-(E^-)^2 \right).
   % & < \left(\frac{1}{\zeta}\right)^{t+1}\left(\operatorname{dist}^2 (x_{t_0},S)-(E^-)^2 \right).
\end{aligned}\end{equation}
%Now, we discuss  the LHS in \eqref{eq:inequality_prod}. 
Denote by $t_1$, the first iterate $t_1 > t_0$ such that $\dist(x_{ t_1},S) < E^-$.
By \Cref{lem:inthetube:const}, $\dist(x_t,S)\leq E^-$, for every $t\geq t_1$, 
 %we have $\dist(x_t,S)\leq E^-$, 
 which implies that 
 %for every $t\geq t_1$ 
 the LHS in \eqref{eq:inequality_prod} is non-positive,
\begin{equation}\label{eq:conv:-}
\left( t\geq t_1\right)\qquad \dist^2(x_{t},S) - (E^-)^2 \leq 0 \end{equation}
On the other side, by our choice,
$\dist(x_t,S) > E^-$
for 
%every $t$ such that 
$t_0\leq t<  t_1$.  
%we have $\dist(x_t,S) > E^-$.
By \Cref{prop:estimation with Et}, for every $  t_0< t<  t_1$, we have
$\operatorname{dist}(x_{t},S) \leq \operatorname{dist}(x_{t-1},S)$,
implying that
\begin{equation}\label{eq:dist:ineq} 
    \operatorname{dist}(x_{t},S) \leq 
 \operatorname{dist}(x_{t-1},S) \leq \dots \leq \operatorname{dist}(x_{t_0},S).
\end{equation}
We can therefore define a lower-bound for $\zeta_{t}$ as follows, for $  t_0\leq t<  t_1$
\begin{equation}\label{eq:zetas}
\begin{aligned}
    \zeta_{t} &= 1-\alpha \rho -\alpha+ \frac{2\alpha \mu}{\operatorname{dist}(x_{t},S)+ E^-} \geq  1-\alpha \rho -\alpha+ \frac{2\alpha \mu}{\operatorname{dist}(x_{t_0},S)+ E^-} = \zeta_{t_0} \\
    &\quad> 1-\alpha \rho -\alpha+ \frac{2\alpha \mu}{E^+ + E^-} = 1.
\end{aligned}
\end{equation}
%where the first inequality stems from \eqref{eq:dist:ineq}, 
%(we include the equality to take into account the case $t=t_0$), 
%the second from \eqref{eq:assuA} and the last equality from the definition of $E^+$ and $E^-$. 
The product in the RHS of \eqref{eq:inequality_prod} can be estimated %by using $\zeta_{t_0}$ 
as follows
\begin{equation}
     \left(\prod_{s=t_0}^{t}\frac{1}{\zeta_{s+1}}\right) \leq  \left(\frac{1}{\zeta_{t_0}}\right)^{t-t_0+1},
\end{equation}
implying that 
 \begin{equation}\label{eq:conv:+}
 \left (  t_0\leq t<  t_1 \right) \qquad    \dist^2(x_{t},S) - (E^-)^2 \leq \left(\frac{1}{\zeta_{t_0}}\right)^{t-t_0}\left(\operatorname{dist}^2 (x_{t_0},S)-(E^-)^2 \right).  
 \end{equation}
 To conclude, we combine  \eqref{eq:conv:-} (for $t > {t_1}$) with \eqref{eq:conv:+} (for $  t_0\leq t<  t_1$) to obtain
 %into the following inequality that holds 
 %for every $t\geq t_0$ 
 %by taking as an upper bound the positive upper bound in \eqref{eq:conv:+}:
 \begin{equation}\label{eq:convergence_inequality}
(\forall t\geq t_0)   \qquad \dist^2(x_{t},S) - (E^-)^2 \leq \left(\frac{1}{\zeta_{t_0}}\right)^{t-t_0}\left(\operatorname{dist}^2 (x_{t_0},S)-(E^-)^2 \right). 
 \end{equation}
 Passing to the $\limsup$ as $t\to \infty$ we conclude that
 \begin{equation}
     \begin{split}
         \limsup_{t\to\infty} \left(\dist^2(x_{t},S) - (E^-)^2 \right)\leq \limsup_{t\to\infty} \left(\frac{1}{\zeta_{t_0}}\right)^{t-t_0}\left(\operatorname{dist}^2 (x_{t_0},S)-(E^-)^2 \right) \\
          = \left(\operatorname{dist}^2 (x_{t_0},S)-(E^-)^2 \right) \lim_{t\to\infty} \left(\frac{1}{\zeta_{t_0}}\right)^{t-t_0}= 0,
     \end{split}
 \end{equation}
where the last equality is due to the fact that $\zeta_{t_0}>1$ as shown in \eqref{eq:zetas}.  We then conclude that
 \begin{equation}
       \limsup_{t\to\infty} \dist(x_{t},S) \leq E^-.
 \end{equation}
 
To show the second part of the statement, we notice that by combining \eqref{eq:assu1} %holds for all $t\geq t_0$, then combining it 
with \eqref{eq:theo:conv_rate}, we get
\begin{equation}\label{eq:dist_e_}
    (\forall t\geq t_0) \qquad 0 <\dist^2(x_{t},S)-(E^-)^2 \leq \left(\frac{1}{\zeta}\right)^{t-t_0}( \dist^2(x_{t_0},S)-(E^-)^2),
\end{equation} 
with the RHS non-negative and monotonically decreasing to $0$ as $t\to\infty$. By \eqref{eq:assu1}, the LHS is positive, and \eqref{eq:dist_e_} implies that
\begin{equation*}
    \lim_{t\to \infty} \operatorname{dist}(x_{t},S)  = E^-,
\end{equation*}
which concludes the proof.
\end{proof}

 %\Cref{theo:Et constant convergence} describes the behaviour of $(\dist(x_t,S))_{t\in\N}$ when \eqref{eq:assu1} holds. Moreover,  
{By \Cref{theo:Et constant convergence}, the distance of the sequence $(x_n)_{n\in \N}$ to the solution set is bounded in the neighbourhood of the solution set. Condition \eqref{eq:assu1}  ensures that the distance of the sequence $(x_n)_{n\in\N}$  to the solution set converges. 

Moreover, if condition \eqref{eq:assu1} is satisfied, then we can obtain the strong convergence of the sequence $(x_t)_{t\in\N}$.}
 %as shown in the following corollary.
\begin{corollary}\label{cor:eps:strong_conv}
    Under the assumptions of \Cref{theo:Et constant convergence}, if \eqref{eq:assu1} holds, then $(x_t)_{t\in\N}$ converges. 
\end{corollary}
\begin{proof}
By \eqref{eq:eps distance square estimate with xi}, (see \Cref{sec:prop:estimation with Et}),
there exists $\zeta_{t+1}>1$ such that
\begin{equation}
\zeta_{t+1} \left(\operatorname{dist}^2 (x_{t+1},S) -(E^-)^2\right) +(1-\alpha L_g) \Vert x_t -x_{t+1}\Vert^2 \leq  \operatorname{dist}^2 (x_{t},S)-(E^-)^2. 
\label{eq: dist const param with xt1}
\end{equation}
By \eqref{eq:assu1}, the first term of the LHS of \eqref{eq: dist const param with xt1} is positive, which implies
\begin{equation}
(1-\alpha L_g) \Vert x_t -x_{t+1}\Vert^2 \leq  \operatorname{dist}^2 (x_{t},S)-(E^-)^2. 
\end{equation}
By  \Cref{theo:Et constant convergence}, the RHS can be estimated with a constant $\zeta >1$
\begin{equation}
(1-\alpha L_g) \Vert x_t -x_{t+1}\Vert^2 \leq \left(\frac{1}{\zeta}\right)^{t-t_0} 
 \left[\operatorname{dist}^2 (x_{t_0},S)-(E^-)^2 \right].
 \label{eq: xt const param zeta}
\end{equation}
As both sides of \eqref{eq: xt const param zeta} are non-negative for all $t\geq t_0$, %taking the square root and summing over $t$ to some $T>t_0$ yields
\begin{equation}\label{eq:sum}
\sqrt{1-\alpha L_g} \sum_{t\geq t_0}^T\Vert x_t -x_{t+1}\Vert \leq 
 \sqrt{\operatorname{dist}^2 (x_{t_0},S)-(E^-)^2 } \sum_{t\geq t_0}^T \left(\frac{1}{\sqrt{\zeta}}\right)^{t-t_0}.
\end{equation}
%Since  $\zeta>1$, 
We  have $\sqrt{\zeta}>1$
%. By assumption, 
and $\sqrt{1-\alpha L_g}>0$. When  $T\to\infty$, the RHS of \eqref{eq:sum} is finite, which implies that 
\begin{equation}
\sum_{t\geq t_0}^{\infty} \Vert x_t -x_{t+1}\Vert < +\infty.
\end{equation}
By virtue of {\cite[Theorem 1]{Bolte2014_Proximal}}, we conclude that $(x_t)_{t\in\N}$ is a Cauchy sequence so it converges.
\end{proof}

\section{Convergence to global solutions in the exact case}\label{sec:convergence:exact}

 In the case of exact proximal computations, we can improve \Cref{theo:Et constant convergence} by admitting a larger upper bound for the step size $\alpha$. \Cref{lem:inthetube:const} and \Cref{prop:estimation with Et} take the form presented in the following lemma. 
\begin{lemma}
\label{cor:1st result of inexact case}
 Let $\X$ be a Hilbert space, $\fonc{f,g}{\X}{(-\infty,+\infty]}$ satisfy \Cref{assu:problem} and  $f+g$ satisfy the sharpness condition \eqref{eq:sharpness:def} with constant $\mu>0$. Let $\left(x_t\right)_{t\in\N}$ be  generated by \Cref{alg:epsFB} with $\varepsilon = 0$ and  the step size  $\alpha>0$ satisfying the condition
\begin{equation*}
(\forall t\in\N)\qquad \alpha < \min \left\{ \frac{1}{L_g} , \frac{1}{\rho}\right\}.
\end{equation*}
For every $t\geq 0$, either $\operatorname{dist}(x_{t+1},S)=0$ or 
%the following inequality 
%\begin{equation*}
$\zeta_{t+1} \operatorname{dist}^2(x_{t+1},S) \leq \operatorname{dist}^2(x_t,S)$,
%\end{equation*}
where 
\begin{equation}
\label{eq:xi_exact}
\zeta_{t+1} =  1  - \alpha \rho +\frac{2\alpha \mu}{\dist (x_{t+1},S)}.
\end{equation}
Moreover, if $\dist(x_{t_0},S) <\frac{2\mu}{\rho}$, for some $t_0\in\N$,  then    %following condition is satisfied for all $t\geq t_0$.

\begin{equation}\label{eq:Inequality2murho}
   \dist(x_{t},S) < \frac{2\mu}{\rho},
\end{equation}
and $\zeta_{t+1} > 1$ for all $t\geq t_0$.
\end{lemma}
\begin{proof}
%See Appendix \ref{app:cor:1st result of inexact case}
%\ref{cor:1st result of inexact case}
 The proof follows the same steps as Lemma 6.1 starting from \eqref{eq:eps:FB_estimate} of Proposition 5.1. Notice that there is no $\varepsilon$ term in this case.
\end{proof}
In the exact case, with 
%it is possible to consider a 
%step size 
%varying in a range 
$\alpha_t \in [\underline{\alpha}, \overline{\alpha}]\subset \left(0, \min\left\{\frac{1}{\rho}, \frac{1}{L_g}\right\}\right)$ and obtain a similar result as %the one in 
\Cref{theo:Et constant convergence}.

\begin{proposition}[\textbf{Strong convergence of the sequence in the exact case}]\label{prop:convergence_convex:new}
  Let $\X$ be a Hilbert space and $\fonc{f,g}{\X}{(-\infty,+\infty]}$ satisfy \Cref{assu:problem} and the function $f+g$ satisfy the sharpness condition \eqref{eq:sharpness:def} with constant $\mu>0$. Let $(x_t)_{t\in\N}$ be  generated by \Cref{alg:epsFB} with $\varepsilon=0$ and  step sizes ${\alpha_t \in [\underline{\alpha},\overline{\alpha}]\subset (0, \min\{\frac{1}{\rho}, \frac{1}{L_g}\})}$. Assume  there exists a $t_0\in \N$ such that $\operatorname{dist}(x_{t_0},S)<\frac{2\mu}{\rho}$.
 Then 
%the following holds 
\begin{equation}\label{eq:dist_to_zero}
  \lim_{t\to \infty} \operatorname{dist}(x_{t+1},S)=0
\end{equation}
 and the sequence $(x_t)_{t\in\N}$ converges strongly to a point $x^*\in S$.
\end{proposition}
\begin{proof}
By \Cref{cor:1st result of inexact case},  for every $t\geq t_0$, there exists $\xi_{t+1} =  1 -\alpha_t \rho +\frac{2\alpha_t \mu}{\dist (x_{t+1},S)}> 1$ such that 
\begin{equation}\label{eq:ineq:xi:exact}
    \xi_{t+1} \dist^2(x_{t+1},S) \leq \dist^2(x_{t},S) . 
\end{equation}
Given the lower-bound for $\alpha_t$, we now define
\begin{equation}
\label{eq:zeta_uniform}\xi:=1+\underline{\alpha}\left(\frac{2\mu}{\dist(x_{t_0},S)}-\rho\right)>1,
\end{equation}
which yields the following inequality for every $t\geq t_0$
\begin{align}
\xi_{t+1} & =1+\alpha_{t}\left(\frac{2\mu}{\operatorname{dist}\left(x_{t+1},S\right)}-\rho\right)
  >1+\alpha_{t}\left(\frac{2\mu}{\dist(x_{t_0},S)}-\rho\right) \geq1+\underline{\alpha}\left(\frac{2\mu}{\dist(x_{t_0},S)}-\rho\right)=\xi.
\end{align}
Combining the latter with \eqref{eq:ineq:xi:exact}, we have that for every $t\geq t_0$
\begin{equation}
% \operatorname{dist}^2\left(x_{t},S\right)\ge \xi_{t+1}\operatorname{dist}^2\left(x_{t+1},S\right)> \xi\left(\underline{\alpha}\right) \operatorname{dist}^2\left(x_{t+1},S\right),
\xi\operatorname{dist}^2\left(x_{t+1},S\right) < \xi_{t+1}\operatorname{dist}^2\left(x_{t+1},S\right) \leq \operatorname{dist}^2\left(x_{t},S\right)
\end{equation}
hence the following \emph{contraction} inequality holds
\begin{equation}
\label{eq: linear convergence}
\operatorname{dist}^2\left(x_{t+1},S\right) < \frac{1}{\xi}\operatorname{dist}^2\left(x_{t},S\right),
\end{equation}
which by recursiveness yields that for every $t\geq t_0$
\begin{equation}
\label{eq: linear convergence:2}
\operatorname{dist}^2\left(x_{t+1},S\right) < \left(\frac{1}{\xi}\right)^{t+1-t_0}\operatorname{dist}^2\left(x_{t_0},S\right).
\end{equation}
For $t\to \infty$ we obtain \eqref{eq:dist_to_zero} and this concludes the proof of the first part of the statement.\\

To show the strong convergence of  $(x_t)_{t\in\N}$, notice that, in the case of exact proximal computation, \eqref{eq:eps:FB_estimate} from \Cref{prop:eps:FB estimate g-convex} can be further estimated 
\begin{align*}
%\label{eq:ineq_without_eta_exact}
\left(f+g\right)\left(\overline{x}\right)-\left(f+g\right)\left(x_{t+1}\right) 
   \geq  \left(\frac{1}{2\alpha } - \frac{L_g}{2}\right)\| x_t - x_{t+1}\|^2 +\left(\frac{1}{2\alpha } -\frac{\rho}{2}\right)\left\Vert \overline{x} -x_{t+1}\right\Vert ^{2} -\frac{1}{2\alpha }\|x_t-\overline{x}\|^2,    %\end{multline}
   \end{align*}
   with $\overline{x}\in S$. Following the proof of \Cref{lem:inthetube:const}, we estimate the LHS by \eqref{eq:sharpness:local:def}

\begin{align*}
   -\mu \operatorname{dist}(x_{t+1},S) + \frac{1}{2\alpha_{t}}\operatorname{dist}^2(x_t,S) \geq\left(\frac{1}{2\alpha_{t}}-\frac{\rho}{2}\right)\operatorname{dist}^2(x_{t+1},S)
    +\left(\frac{1}{2\alpha_{t}}-\frac{L_{g}}{2}\right)\left\Vert x_{t}-x_{t+1}\right\Vert ^{2}.
%\end{aligned}\end{equation}
\end{align*}
Multiplying  both sides by $2\alpha_t$ we get
\begin{equation}\label{eq:the_inequality_above}
\begin{aligned}
   -2\alpha_t\mu \operatorname{dist}(x_{t+1},S) + \operatorname{dist}^2(x_t,S) \geq\left({1}-{\rho}\alpha_t\right)\operatorname{dist}^2(x_{t+1},S)
    +\left(1-{L_{g}}\alpha_t\right)\left\Vert x_{t}-x_{t+1}\right\Vert ^{2}.
\end{aligned}\end{equation}
Let now $t\geq t_0$. From \eqref{eq:the_inequality_above} we infer
\begin{align}
\label{eq:ineq:exactcase}
&\left(1-\alpha_{t}L_{g}\right)\left\Vert x_{t}-x_{t+1}\right\Vert ^{2}  \leq \operatorname{dist}^2(x_{t},S)-\operatorname{dist}^2(x_{t+1},S)+\underbrace{\alpha_{t}\rho \operatorname{dist}^2(x_{t+1},S)-2\mu\alpha_{t}\operatorname{dist}(x_{t+1},S)}_{<0}
% & \leq \operatorname{dist}^2(x_{t},S)
\end{align}
where for every $t\geq t_0$, by virtue of \Cref{cor:1st result of inexact case} we have
\begin{equation}
    \alpha_t\operatorname{dist}(x_{t+1},S)\left(\rho\operatorname{dist}(x_{t+1},S) - 2\mu\right) <0.
\end{equation}
Since all the terms depending on $x_{t+1}$ in the RHS of \eqref{eq:ineq:exactcase} are negative, 
%we infer the following upper bound for $\left(1-\alpha_{t}L_{g}\right)\left\Vert x_{t}-x_{t+1}\right\Vert ^{2}$ in the form
\begin{equation}
(t\geq t_0) \qquad \left(1-\alpha_{t}L_{g}\right)\left\Vert x_{t}-x_{t+1}\right\Vert ^{2} \leq  \operatorname{dist}^2(x_{t},S)
\end{equation}
Given the upper-bound for $\alpha_t $, %we have 
\sloppy we obtain the following lower bound for $\left(1-\alpha_{t}L_{g}\right)\left\Vert x_{t}-x_{t+1}\right\Vert ^{2}$:
\begin{equation}
(\forall t \in\N) \qquad \left(1-\overline\alpha L_{g}\right)\left\Vert x_{t}-x_{t+1}\right\Vert ^{2} \leq  \left(1-\alpha_{t}L_{g}\right)\left\Vert x_{t}-x_{t+1}\right\Vert ^{2}.
\end{equation}
By taking $\xi>1$ from \eqref{eq:zeta_uniform},  for every $t> t_0$ we have
\begin{equation}
\left\Vert x_{t}-x_{t+1}\right\Vert \leq\frac{\operatorname{dist}(x_{t},S)}{\sqrt{1-\alpha_{t}L_{g}}}\leq \frac{\operatorname{dist}(x_{t},S) }{\sqrt{1-\overline{\alpha}L_g}}< \frac{1 }{\sqrt{1-\overline{\alpha}L_g}}\left(\frac{1}{\sqrt{\xi}}\right)^{t-t_0}\operatorname{dist}(x_{t_0},S),
\end{equation}
where the last inequality is a consequence of \eqref{eq: linear convergence:2}. In conclusion, 
%we obtain 
\begin{equation}
    \sum_{t=0}^{\infty} \|x_{t}- x_{t+1}\| < +\infty
\end{equation}
 which means that $(x_{t})_{t\in\N}$ is a Cauchy sequence, hence it converges to some $x^*\in\X$ (see {\cite[Theorem 1]{Bolte2014_Proximal}}). Eventually,  $x^*\in S$ because $S$ is closed 
 %(as it is the level set of a lower semicontinuous function) 
 and $\operatorname{dist}(x_t,S)\rightarrow 0$ as $t\rightarrow +\infty$.
\end{proof}
}
{
\begin{remark}[Strong Convergence in the exact convex case]
Under the sharpness condition  \eqref{eq:sharpness:def}, for convex function ($\rho=0$) with exact proximal calculation ($\varepsilon=0$), starting from \eqref{eq:distance with eta and eps} one can arrive at 
\begin{equation*}
    2\alpha\mu \dist(x_{t+1},S) \leq \dist^2(x_t,S) - \dist^2(x_{t+1},S).
\end{equation*}
Taking the sum for $t=0,\dots,T$ yields
\begin{equation*}
    2\alpha\mu \sum_{t=0}^{T}\dist(x_{t+1},S) \leq \sum_{t=0}^{T}\left(\dist^2(x_t,S) - \dist^2(x_{t+1},S)\right) = \dist^2(x_0,S).
\end{equation*}
Hence, for $T\to \infty$, 
\begin{equation*}
    2\alpha\mu \sum_{t=0}^{\infty}\dist(x_{t+1},S) \leq \dist^2(x_0,S) < +\infty,
\end{equation*}
which implies $\dist(x_t,S) \to 0$ as $t\to\infty$. In conclusion, we obtain   \cite[Theorem 3.4(d)]{combettes2005signal}. %which is the strong convergence of the sequence $(x_t)_{t\in\N}$ to a solution.
\end{remark}
}

 \section{Applications}\label{sec:feasibility}
{
% One of these schemes is 
% \begin{equation}
%     \minimize{\x\in \H}\,\frac{\dist^2 (\x,Q)}{2}+\frac{\dist^2 (\x,C)}{2}.
% \end{equation}

In this section, we cast binary-constrained problems in $\H = \R^n$ as instances of ~\eqref{eq: ourproblem}. In particular, we will be considering the structured problem 
\begin{equation}
    \label{FP0}
    \minimize{x\in \H}\,f(x)+\frac{\dist^2 (\mathbf{A}x,D)}{2},
\end{equation}
where $f$ is a sharp and weakly convex function, $D\subseteq \R^m$ is a closed convex set and $\fonc{\mathbf{A}}{\R^n}{\R^m}$ is a linear operator.
Problem~\eqref{FP0} satisfies the assumptions of \Cref{theo:Et constant convergence} and \Cref{prop:convergence_convex:new}. In particular, by \Cref{theo:sharpness 1} -- which we will soon present --, the objective of~\eqref{FP0} is sharp.\\

% We denote the unit sphere centred in the origin with $\mathbb {S}$. 
% Gabriele's example shows that the proof in Feasibility.tex is somehow wrong, as we can see from the following image.
%  \begin{minipage}[c]{0.47\textwidth}
%     \includegraphics[width=0.4\textwidth]{Feaswrong.png}
%     \label{fig:test}
%   \end{minipage}\hfill
%   \begin{minipage}[c]{0.5\textwidth}
% $Dist (\x,S)$ in blue, $Dist (\x,S) +2\|\x\|^2$ in red.
% \label{fig:Feaswrong}
%   \end{minipage}
%  To model our feasibility problem, we define the function $f:\H\rightarrow \R$ as

% \begin{equation}\label{eq:5:f}
%  \quad    \left(\forall \x \in \H\right)\quad  f(\x) := |\|\x\|^2-1|,
% \end{equation}
% and we highlight that $\min_{\x\in\H}\, f(\x)=0$ and $\argmin_{\x\in\H}\, f(\x)=\mathbb {S}$

To model the binary constraints, let us consider the function $\fonc{f}{\mathbb{R}^n}{\mathbb{R}}$ defined as
\begin{equation}\label{eq:binaryfunction}
 \left(\forall x\in \H\right)\quad   f({x}) = {\sum_{i=1}^n |x_i^2-1|},   
\end{equation}
which corresponds to a separable sum of the function from \Cref{ex:critical_points} with $a=-1$ and $b=1$. Function $f$ is a (globally) sharp and weakly convex function with the set of minimisers \[{S = \{x=(x_1,\dots,x_n)\in\R^n\mid x_i \in\{-1,1\} \;\text{for}\; i\in \{1,\dots,n\}\} = {\{-1,1\}}^n}.\]

\Cref{fig:comparisons} illustrates a comparison between function $f$ and the distance from $ {S}$ when $\H = \R$. 
%and for this reason, this f
Function $f$ can be seen as a non-smooth counterpart to the function 
 ${ x\mapsto  {\sum_{i=1}^n x_i^2{(x_i-1)}^2} }$ used in ~\cite{Xu2021DoublyGraduated} and it is similar to the function ${ x\mapsto  {\sum_{i=1}^n x_i{(1-x_i)}} } + \iota_{[0,1]^n}$ used in ~\cite{bonettini2023abstract} with $x\in\R^n$ to promote binary integer solutions with values in ${\{0,1\}}^n$. To the best of our knowledge, this is the first work where function $f$ is used in such a context.\\

\begin{remark}
\label{rmk: prox form of x2 -1}
Since function $f$ is defined as a separable sum of functions of the form $\overline f = |(\cdot)^2-1|$, for $\alpha>0$, the proximity operator of $f$ is defined as
\begin{equation}
(\forall y\in \H) \qquad   \prox_{\alpha f}(y) = \bigtimes_{i=1}^n \prox_{\alpha \overline f}(y_i), 
\end{equation}
    We give an explicit form of proximal operator for the $2$-weakly convex function $\overline f = |( \cdot)^2 -1|$. Consider $0<\alpha <1/2$ and $y \in \R$. Our aim is to solve the minimisation problem
    \begin{equation}
        \argmin_{x\in \R} f(x) +\frac{1}{2\alpha} ( x-y)^2.
    \end{equation}
    Let $x\in \R$ be its unique minimiser. 
    %We have the following:
 If $|x| >1$, then $x= \frac{y}{2\alpha +1}$; if $| x| <1$, then $x = \frac{y}{1-2\alpha}$; if $ x =1$, we have that the first-order optimality condition reads as $ 0 \in [-2\alpha,2\alpha] + (1-y)  $, implying $y \in [1-2\alpha,1+2\alpha]$; if $ x = -1$, we have that the first-order optimality condition reads as $ 0 \in [-2\alpha,2\alpha] +(-1-y)  $, implying $y \in [-1-2\alpha,-1+2\alpha]$.
      %   \begin{equation}
      % \argmin_{\underset{|x|=1}{x\in\R}}      f(x)+\frac{1}{2\alpha} ( x-y)^2 = \argmin_{{x\in\R}}  \frac{1}{2\alpha} ( x-y)^2 %\geq \frac{(\|y\| -1)^2}{2\alpha}.
      %   \end{equation}
        
      % %   Hence $\x = \frac{y}{\|y\|}$ when $\|y\|\neq 0$. 
      % \begin{equation}
    % 
      % \end{equation}
      In conclusion, 
\begin{equation}\label{eq:prox_of_f}
\quad (\forall y \in \R  )\quad 
    \prox_{\alpha f} (y) = \begin{cases}
        \frac{y}{1+2\alpha} & \text{ if } | y | >1+2\alpha\\
        \frac{y}{1-2\alpha} & \text{ if } | y | <1-2\alpha\\
        \frac{y}{|y|} & \text{ otherwise. } 
    \end{cases}  
\end{equation}
\end{remark}

\begin{figure}
\centering
\resizebox{0.25\textheight}{!}{
\begin{tikzpicture}
\begin{axis}[
    axis lines = center,
    % xlabel = $\|\x\|$,
    %ylabel = {$f(\x)$},
    legend style={at={(0.5,1.2)}, anchor=north,legend columns=-1}
]

\addplot [
    domain=-2:2, 
    samples=300, 
    color=black,
    ]
    {abs(x^2 -1)};
\addlegendentry{$|x^2 -1|$}
\addplot [dashed,
    domain=-2:2, 
    samples=300, 
    color=black,
]
{abs(abs(\x)-1)};
\addlegendentry{$||x| -1|$}
\end{axis}
\end{tikzpicture}}
\caption{Comparison between function $f$ (continuous line) and the distance from $S$ (dashed line) for $\H = \R$.}
    \label{fig:comparisons}
\end{figure}
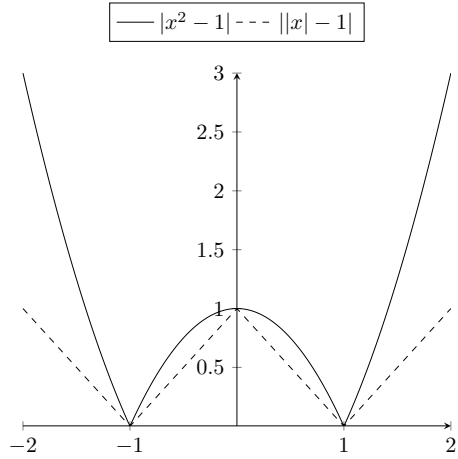

The  properties of the function  ${\dist^2(\cdot,D)}/{2} $, 
 are summarised in the following lemma.

\begin{lemma}[{\cite[Corollary 12.30, Example 13.5]{Bauschke2017}}]\label{lem:dist}
    Let $D\subseteq \H$ be a convex set. Then 
    \begin{equation}\label{eq:dist_func:diff_conv}
 \quad \left(\forall x \in \H\right)\quad  
    \frac{\dist^2 (x,D)}{2} = \frac{\|x\|^2}{2} - \left( \frac{\|\cdot\|^2}{2} + \iota_D\right)^*(x).  
\end{equation}
If, in addition, $D$ is closed, then  $\dist^2(\cdot,D)$ is Fréchet differentiable on $\H$ and 
    \begin{equation} \left(\forall x \in \H\right) \quad \nabla \frac{\dist^2 (x,D)}{2} = x - \proj_D(x).  
    \end{equation}
\end{lemma}
 %We highlight
%  Note that $\nabla g$ is Lipschitz continuous 
%  %as the sum of two Lipschitz continuous functions. With the explicit form of $f$ and $g$, 
%  and the feasibility problem (FP) becomes

% \begin{equation}
% \label{FP}
% \minimize{x\in\H}\, \sum |x_i^2-1|+\frac{\dist^2 (\mathbf{A}x,D)}{2},
% \end{equation}
% where the objective function is non-convex: according to \Cref{lem:dist}, the problem can be recast as 
% \begin{equation}
%     \minimize{x\in\H}\, \sum |x_i^2-1|+ \frac{\|x\|^2}{2} - \left( \frac{\|\cdot\|^2}{2} + \iota_D\right)^*(x),
% \end{equation}
% where  $x\mapsto |\|x\|^2-1|+ \frac{\|x\|^2}{2} $ is $1$-weakly convex by \Cref{lem:weak} and function ${x\mapsto- \left( \frac{\|\cdot\|^2}{2} + \iota_D\right)^*(x)}$ is the negative of a conjugate function, 
% hence concave 
% %since a Fenchel conjugate function is always convex 
% \textup{\cite[Proposition 13.13]{bauschke2017convex}}.

The sharpness of $f+g$, which is the core of our convergence result, 
holds by virtue of the following theorem, where $h_1 = f$ and $h_2 = g$.
\begin{theorem}\label{theo:sharpness 1}
    Let functions $\fonc{h_1,h_2}{\H}{\R}$ be proper and let the following assumptions be satisfied:
    \begin{enumerate}[parsep = 0pt, itemsep=5pt, topsep=5pt, label=\textup{(\roman*)}]
        \item the global minimiser set $S$ of $h_1+h_2$ is equal to (or contained in) $S_1\cap S_2$, if both non-empty or $S\subset S_1$ where $S_1,S_2$ are the sets of the minimisers of $h_1$ and $h_2$ respectively;
        \item the optimal value of $h_1+h_2$ is zero;
        \item $h_1(x)\geq 0$, $h_2(x) \geq 0$ for all $x$ in a neighborhood of $S$;
        \item $h_1$ is sharp with respect to $S$ locally or globally with constant $\mu>0$ and ${\inf_{x\in\H} h_1(x) =0}$.
    \end{enumerate}
    Then $h_1+h_2$ is locally or globally sharp with respect to $S$.
\end{theorem}

\begin{proof}
By {(iv)} there exists $\delta_1> 0$ such that 
for every $x\in B(S,\delta_1)$ we have $h_1(x) \geq \mu \dist(x,S)$.
By (iii), there exists $\delta_2> 0$ such that for every $ x\in B(S,\delta_2)$ function $h_2$ is non-negative, hence by taking $\delta=\min\{\delta_1,\delta_2\}$ we have 
\begin{equation}
\quad   \left(\forall  x\in B(S,\delta)\right)\quad  h_1(x) + h_2(x) \geq \mu \dist  (x,S).  
\end{equation}
By (i) and (ii), the above inequality corresponds to the local sharpness of function $h_1+h_2$.
\end{proof}

In the next section, we will use the considered framework to model and solve a discrete tomography problem.
} \subsection{Discrete Tomography}
\label{sec:discrete_tomography}
% \paragraph{A sharp weakly convex function to promote binary solutions}
%\paragraph{Binary Tomography}
Binary Tomography (BT) is a special case of Discrete Tomography (DT) which aims at reconstructing binary images starting from a limited number of their projections \cite{SCHULE2005_DiscreteTomography, Kadu2029ConvexBinaryTomography}. 
%With (DT), images can be reconstructed relying on a much smaller data set than for the Computerized Tomography (CT). 
%The problem can be mathematically cast as follows:
The image $\overline x \in \R^n$ is represented as a grid of $n=n_1\times n_2$ pixels taking values $x_j \in \{-1,1\}$ for $j=1,\dots,n$. The projections $y_i$ for $i=1,\dots,m$ are linear combinations of the pixels along $m$ directions. The linear transformation from image to projections is modelled as 
\begin{equation}
    y=\mathbf{A}\overline x + \omega
\end{equation}
where $x_j$ denotes the value of the image in the $j$-th cell, $y_i$ is the weighted sum of the image along the $i$-th ray, the element $\mathbf{A}_{i,j}$ of matrix $\mathbf{A}\in\R^{m\times n}$ is proportional to the length of the $i$-th ray in the $j$-th cell and finally $\omega\in\R^n$ is an additive noise with Gaussian distribution and standard deviation $\sigma>0$. In general, the projection matrix $\mathbf A$ has a low rank, meaning that $\operatorname{rank}(\mathbf A) < \min\{ n,m\}$ \cite{Kadu2029ConvexBinaryTomography}. We cast the problem as
\begin{equation}\label{eq:binary1}
    \text{Find}\quad x\in\{-1,1\}^n\quad\text{such\,that}\quad \|y-\mathbf Ax\|_2 \leq \theta
\end{equation}
for some $\theta>0$, which can be reformulated as
\begin{equation}\label{eq:discrete_tomography_objective}
    \minimize{{x\in {C}}}  F(x) \quad \text{where}\quad {C} = \{x\in\R^n\,|\,\|\mathbf Ax-y\|_2\leq \theta\}
\end{equation}
with $F$ defined as in \eqref{eq:binaryfunction} and the constraint  set ${C}$, represented by the function
  $\dist^2(\cdot,\overline{B}(y,\theta))/2$. In conclusion, we have
\begin{equation}\label{eq:discrete_tomography_objective_final}
    \minimize{{x\in\R^n}}  F(x)  + \frac{\dist^2(\mathbf Ax,\overline{B}(y,\theta))}{2}.
\end{equation}
The sharpness of the objective function is ensured by \Cref{theo:sharpness 1}, assuming that the set ${S}\cap C \neq \emptyset$ and equivalently $S \cap \argmin \frac{\dist^2(\mathbf Ax,\overline{B}(y,\theta))}{2}  \neq \emptyset$ (for an adequate value of $\theta$, the original signal $\overline{x}$ belongs to the intersection). Clearly, for every $x\in \R^n$, $\iota_{\overline{B}(y,\theta)}(x) = \iota_{\overline{B}(0,\theta)}(x-y)$ and 
  \begin{equation}
      \proj_{\overline{B}(y,\theta)}(x) = y + \proj_{\overline{B}(0,\theta)}(x-y) = y + \left( (x-y) - \prox_{\theta \|\cdot\|_2}(x-y)\right) = x - \prox_{\theta \|\cdot\|_2}(x-y).
  \end{equation}
  It follows that, by \Cref{lem:dist}, for every $x\in\R^n$, $\nabla\left( \frac{\dist^2(\mathbf Ax,\overline{B}(y,\theta))}{2}\right) =  \mathbf A^\top\left(  \prox_{\theta \|\cdot\|_2}(\mathbf Ax - y)\right)$
   and $\|\mathbf A\|^2$ is a Lipschitz constant for this gradient.

\begin{remark}
    The model in \eqref{eq:binary1} and \eqref{eq:discrete_tomography_objective} (for appropriate choices of matrix $\mathbf A\in\R^{m\times n}$, vector $y\in\R^m$ and possible additional convex constraints) could be applied to other Binary Quadratic Programs (a special class of QCQP problems having the equality constraint $x_i^2=1$ for every $i\in\{1, \dots, n\}$ ) arising in Computer Vision, such as Graph Bisection, Graph Matching and Image Co-Segmentation (see \cite[Table 2]{Wang2016_BQ} and the references therein).
\end{remark}
% { \color{purple}
\paragraph{Numerical tests} For our simulations we used the MATLAB codes and data from \cite{Kadu2029ConvexBinaryTomography, githubpage}. We reproduced the same synthetic setting: we considered four phantoms (\emph{Apple}, \emph{Lizard}, \emph{Bell}, \emph{Bird}), corresponding to binary images of size 64 x 64 pixels. Operator $\mathbf A$ models an X-ray tomographic scan with 64 detectors and a parallel beam acquisition geometry with four angles (0°, 50°, 100°, 150°). We set $\sigma=0.01$ for the additive noise and $\theta = 10(64\sigma)^2$ in the constraint set $\overline{B}(y,\theta)$. \Cref{fig:discrete_tomography} illustrates, from left to right, the original image and the reconstructions obtained with the Least Squares QR method (LSQR), with the Truncated Least Squares QR method (TLSQR), with the DUAL method proposed in \cite{Kadu2029ConvexBinaryTomography} and finally with \Cref{alg:epsFB} applied to \eqref{eq:discrete_tomography_objective_final}, which we refer to as Continuously Relaxed Binary Tomography (CRBT). All methods are initialised with a vector of zeros.

\begin{figure}[htb]
    \centering
    \begin{tabular}{p{0.16\textwidth}p{0.16\textwidth}p{0.16\textwidth}p{0.16\textwidth}p{0.16\textwidth}}
    \phantom{ccc} ORIG & \phantom{cci} LSQR & \phantom{cc} TLSQR & \phantom{cci} DUAL & \phantom{cci}{CRBT} \\
  \multicolumn{5}{c}{\centering
    \includegraphics[width=0.9\textwidth,trim={0 0 0 13},clip]{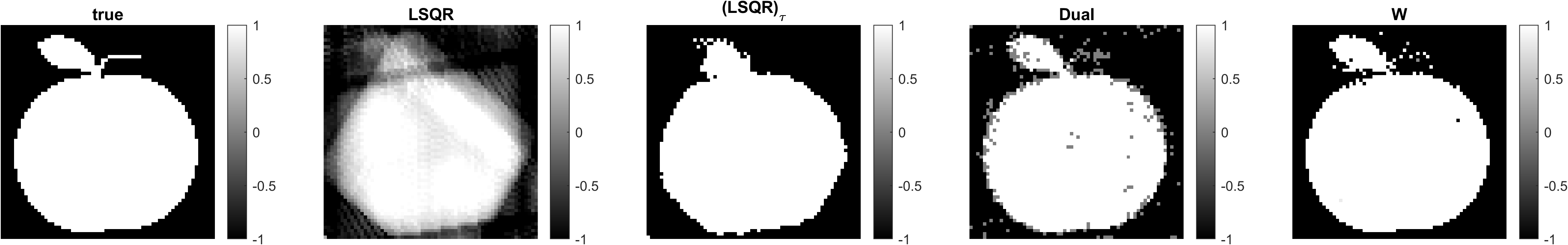}}\\
   \multicolumn{5}{c}{ \includegraphics[width=0.9\textwidth,trim={0 0 0 13},clip]{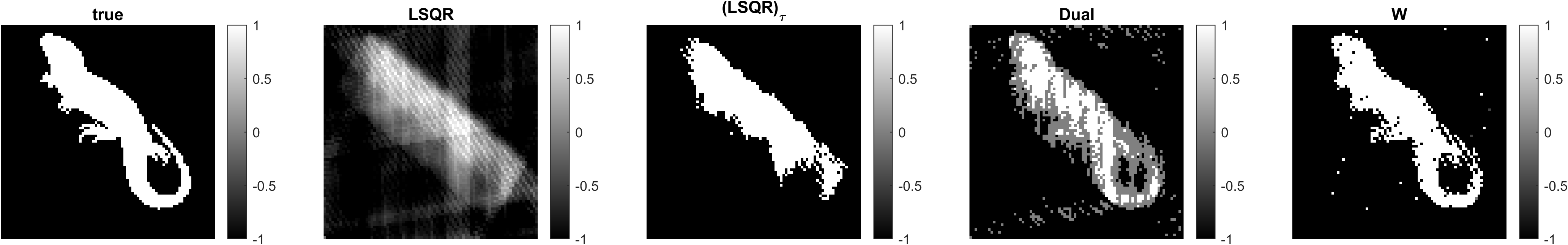}}\\
   \multicolumn{5}{c}{ \includegraphics[width=0.9\textwidth,trim={0 0 0 13},clip]{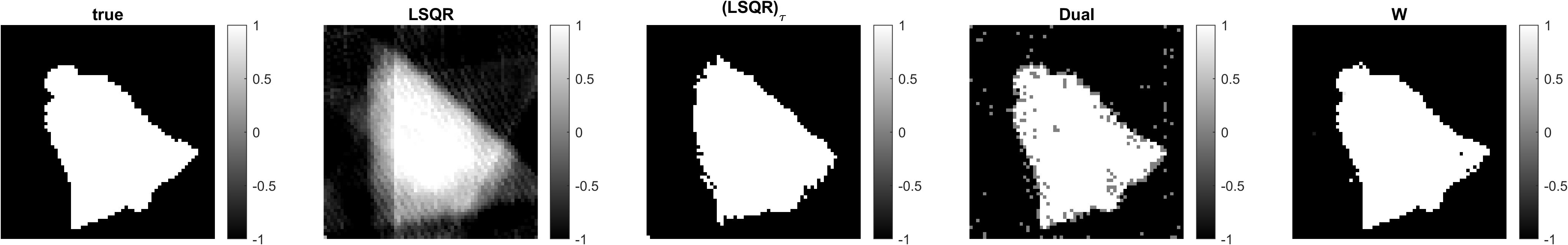}}\\
   \multicolumn{5}{c}{ \includegraphics[width=0.9\textwidth,trim={0 0 0 13},clip]{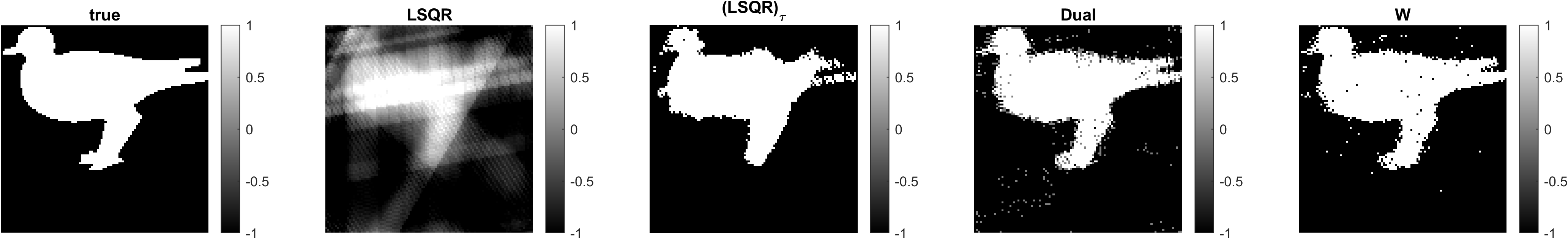}}\\
      \end{tabular}
       \caption{comparisons between different solutions. From left to right: original image, least squares solution, thresholded least squares solution, solution obtained with the dual method proposed in \cite{Kadu2029ConvexBinaryTomography}, solution obtained with our model.}
    \label{fig:discrete_tomography}
\end{figure}
\section{Conclusions}
{We investigate the convergence properties of the exact and the inexact forward-backward algorithms for the minimisation of a function that is expressed as the sum of a weakly convex function and a smooth function with Lipschitz-continuous gradient. In the inexact case, we inferred a convergence result that relies on the hypothesis that the accuracy level $\varepsilon> 0$ for the inexact proximal computation is kept constant throughout all the iterations. 
%To carry out our analysis, we exploited the notion of proximal $\varepsilon$-subdifferential. 
We successfully applied the analysed method to binary tomography.
It will be interesting, in future work, to extend our results so as to take into account non-constant accuracy levels $\varepsilon_t$.\\  } 
\paragraph{Funding}{This work was funded by the European Union's Horizon 2020 research and innovation program under the Marie Sk{\l}odowska-Curie grant agreement No 861137. This work represents only the authors' view, and the European Commission is not responsible for any use that may be made of the information it contains.}}

\bibliographystyle{acm}
\bibliography{references}

\begin{appendices}
\section{Proof of \Cref{cor:sh_versus_kl}}
\label{app:proof_of_equivalence}
{ In order to proceed with the proof of \Cref{cor:sh_versus_kl}, we need the two following results. The first result is presented in the following lemma, whose proof follows from {\cite[Theorem 2.1]{aze2004characterizations}} and {\cite[Remark 12.(ii)]{bolte2010characterizations}. }
\begin{lemma} 
\label{thm:sh_versus_kl}
    Let $\X$ be a Hilbert space and $h\in \Gamma_{\rho}(\X)$, $h\geq 0$, $S = \argmin_{x\in\X} h(x) = [h\leq 0]\neq\emptyset$. Let $r_0 >0$. Then 
    \begin{equation}
     \inf_{x \in [0< h<r_0]}\dist(\partial_{\rho/2} h(x),0) = \inf_{0\leq r < r_0} \inf_{x\in[r < h< r_0]} \frac{h(x) - r}{\dist(x, [h\leq r])}. % \leq \inf_{x\in[0 < h< \beta]} \frac{h(x)}{\dist(x, [h\leq 0])} \leq \frac{h(x)}{\dist(x,[h\leq 0])} 
    \end{equation}
%    for all $x \in [0<h<r_0]$.
\end{lemma}}
The following theorem is inspired by {\cite[Theorem 5.2]{studniarski1999weak}},  originally in the finite-dimensional setting, which can be easily adapted to the general Hilbert space setting.
\begin{theorem} \label{thm:stud_ward}
Let $\X$ be a Hilbert space and $h:\X\rightarrow\mathbb{R}$ be a continuous $\rho$-weakly convex function, $ h\geq 0 $,  $x^{*}\in S = \argmin_{x\in\X} h(x) = [h\leq 0]\neq\emptyset$.  Let $\delta_{0}>0$ and $r_0>0$. If there exists $\mu>0$ such that for every $x\in \overline{B}(x^{*},\delta_{0})\cap[0<h\leq r_0]$
\begin{equation}
    \label{eq:kl1}
 \left(\forall  z\in\partial_{\rho} h(x)\right) \qquad  \|z\|\ge \mu 
\end{equation}
then there exists $\eta>0$ such that for every $ x\in {B}(x^{*},\eta)\cap[0<h<r_0]$
\begin{equation}
    \label{eq:sh1}
    \ \ h(x)-h(x^{*})\ge \mu\ \dist(x,{S})
\end{equation}
\end{theorem}

\begin{proof}[Proof of \Cref{cor:sh_versus_kl}]

\textbf{($(i) \implies (ii)$) } We assume \eqref{eq:prop:sh_loc} holds for all ${x\in  B(x^*, \delta_1)\cap [ 0 < h(x) < r_1]}$. For any $0 < \delta < \delta_1$ we consider $\overline{h} = h + \iota_{\overline{B}(x^*,\delta)}$. By \Cref{thm:sh_versus_kl}, we have
   \begin{equation}
     \inf_{x \in [0< \overline h<r_1]}\dist(\partial_{\rho/2} \overline h(x),0) = \inf_{0\leq r < r_1} \inf_{x\in[r <\overline  h< r_1]} \frac{ \overline h(x) - r}{\dist(x, [\overline h\leq r])} % \leq \inf_{x\in[0 < h< \beta]} \frac{h(x)}{\dist(x, [h\leq 0])} \leq \frac{h(x)}{\dist(x,[h\leq 0])} 
    \end{equation}
For the RHS, we have
  \begin{equation}
   \operatorname{RHS} = \inf_{0\leq r < r_1} \inf_{x\in[r <  h< r_1] \cap \overline{B}(x^*,\delta)} \frac{  h(x) - r}{\dist(x, [ h\leq r])}
   \geq  \inf_{0\leq r < r_1} \inf_{x\in[r <  h< r_1] \cap {B}(x^*,\delta_1)} \frac{  h(x) - r}{\dist(x, [ h\leq r])}
   % \leq \inf_{x\in[0 < h< \beta]} \frac{h(x)}{\dist(x, [h\leq 0])} \leq \frac{h(x)}{\dist(x,[h\leq 0])} 
    \end{equation}
    % and if $SH_loc$ holds, then $RHS > \mu$
   %  and in particular we have
   %  \begin{equation}
   % RHS \geq  \inf_{0\leq r < r_1} \inf_{x\in[r <  h< r_1] \cap {B}(x^*,\delta_1)} \frac{  h(x) - r}{\dist(x, [ h\leq r])}. % \leq \inf_{x\in[0 < h< \beta]} \frac{h(x)}{\dist(x, [h\leq 0])} \leq \frac{h(x)}{\dist(x,[h\leq 0])} 
   %  \end{equation}
    because $\overline{B}(x^*, \delta) \subset{B}(x^*, \delta_1) $.
On the other side, by {\cite[Theorem 2]{bednarczuk2022calculus}}, the LHS is equivalent to 
 \begin{equation}
   \operatorname{ LHS} =  \inf_{x \in [0< h<r_1]}\dist(\partial_{\rho/2}  h(x) + N_{\overline{B} (x^*,\delta)}x,0).
    \end{equation}
where $N_{\overline{B} (x^*,\delta)}$ is the Normal Cone in the sense of convex analysis. For any $0<\delta_2 < \delta$ and for any $0<r_2<r_1$ we have
\begin{equation}
   \operatorname{ LHS} \leq   \inf_{x \in [0< h\leq r_2] \cap \overline B(x^*,\delta_2)}\dist(\partial_{\rho/2}  h(x)+ N_{\overline{B}(x^*,\delta)} x ,0) = \inf_{x \in [0< h\leq r_2] \cap \overline B(x^*,\delta_2)}\dist(\partial_{\rho/2}  h(x),0)
    \end{equation}
    where the last equality stems from the fact that  $N_{\overline{B}(x^*,\delta)}x=0$ for every $x \in B(x^*,\delta)$. In conclusion, we obtain
    \begin{equation}
  \inf_{x \in [0< h\leq r_2] \cap \overline B(x^*,\delta_2)}\dist(\partial_{\rho/2}  h(x),0) \geq \inf_{0\leq r < r_1} \inf_{x\in[r <  h<  r_1] \cap {B}(x^*,\delta_1)} \frac{  h(x) - r}{\dist(x, [ h\leq r])}
    \end{equation}   
 %     \begin{equation}
 % \inf_{x \in [0< h\leq r_2] \cap \overline B(x^*,\delta_2)}\dist(\partial_{\rho}  h(x),0)  \geq  \inf_{x \in [0< h<r_1] \cap \overline B(x^*,\delta_2)}\dist(\partial_{\rho}  h(x),0) 
 %    \end{equation}
and by \eqref{eq:prop:sh_loc} we get
\begin{equation}
\inf_{0\leq r < r_1} \inf_{x\in[r <  h<  r_1] \cap {B}(x^*,\delta_1)} \frac{  h(x) - r}{\dist(x, [ h\leq r])} \geq \mu>0,
    \end{equation} 
which implies
%\begin{equation}
  $\inf_{x \in [0< h\leq  r_2] \cap \overline B(x^*,\delta_2)}\dist(\partial_{\rho/2}  h(x),0) \geq \mu >0$,
%\end{equation}
which is \eqref{eq:prop:kl_loc} with function $\varphi = \frac{1}{\mu}\operatorname{Id}$.\\

\textbf{($(i) \implies (ii)$) } The second part of the assertion follows from \Cref{thm:stud_ward}.
\end{proof}

\section{Proof of \Cref{lem:inthetube:const}}\label{sec:lem:inthetube:const}
%Let us present an auxiliary result about the behavior of $\dist(x_t,S)$ which is needed in the proof of \Cref{lem:inthetube:const}

\begin{proof}
{
Take an arbitrary but fixed $t\in\mathbb{N}$ and let $\operatorname{dist}(x_{t+1},S)>0$.
We take any $\overline{x}\in S$  and  apply \eqref{eq:eps:FB_estimate} from \Cref{prop:eps:FB estimate g-convex} for $x=\overline{x}$, which yields the following inequality:
\begin{equation}\label{eq:eps:FB_estimate:2}
\begin{aligned}
    (f+g)(\overline{x})  - (f+g)&(x_{t+1}) \geq \innerprod{\frac{x_t - x_{t+1} }{\alpha }}{\overline{x}-x_{t+1}} \\
    &  - \frac{\rho}{2}\|\overline{x}-x_{t+1}\|^2 - \frac{L_g}{2}\|x_{t}- x_{t+1}\|^2 - \varepsilon  - \sqrt{\frac{2\varepsilon }{\alpha }}\| \overline{x} - x_{t+1}\|.\\
    \end{aligned}
\end{equation}
%where $\operatorname{dist}(x_{t+1},S)>0$.
{ By using the identity
\begin{equation*}
    \left(\forall (w,y,z)\in\X^3\right) \qquad   \innerprod{w-y}{y-z} = \frac{1}{2}\|w-z\|^2 - \frac{1}{2}\|w-y\|^2 - \frac{1}{2}\|y-z\|^2, 
\end{equation*}
which follows from \cite[Lemma 2.12]{Bauschke2017}, we obtain}
\begin{equation}\label{eq:ineq_without_eta}
\begin{aligned}
&\left(f+g\right)\left(\overline{x}\right)-\left(f+g\right)\left(x_{t+1}\right) \geq  \frac{1}{2\alpha }\left( -\|x_t-\overline{x}\|^2 +\|x_{t+1}- \overline{x}\|^2  + \| x_t - x_{t+1}\|^2\right) \\
    & \qquad-\frac{\rho}{2}\left\Vert \overline{x} -x_{t+1}\right\Vert ^{2}-\frac{L_{g}}{2}\left\Vert x_{t}-x_{t+1}\right\Vert ^{2} - \varepsilon  - \sqrt{\frac{2\varepsilon }{\alpha }}\left\Vert \overline{x}-x_{t+1}\right\Vert  \\
    & \qquad= \left(\frac{1}{2\alpha } - \frac{L_g}{2}\right)\| x_t - x_{t+1}\|^2 +\left(\frac{1}{2\alpha } -\frac{\rho}{2}\right)\left\Vert \overline{x} -x_{t+1}\right\Vert ^{2}-\frac{1}{2\alpha }\|x_t-\overline{x}\|^2 - \varepsilon   - \sqrt{\frac{2\varepsilon }{\alpha }}\left\Vert \overline{x}-x_{t+1}\right\Vert  .\\
    \end{aligned}
    \end{equation}
Using Young's Inequality, we have
\begin{equation}\label{eq:ineqeta}
     - \sqrt{\frac{2\varepsilon }{\alpha }}\left\Vert \overline{x}-x_{t+1}\right\Vert  \geq  - \frac{1 }{2 }\|\overline{x}- x_{t+1}\|^2 - \frac{\varepsilon }{\alpha }.
  \end{equation}
Plugging \eqref{eq:ineqeta} back into \eqref{eq:ineq_without_eta} yields
% , by combining \eqref{eq:ineq_without_eta} and \eqref{eq:ineqeta}, we obtain
\begin{equation}\label{eq:ineq_with_eta}
\begin{aligned}
&\qquad\qquad \left(f+g\right)\left(\overline{x}\right)-\left(f+g\right)\left(x_{t+1}\right) \geq \\
    &  \left(\frac{1}{2\alpha } - \frac{L_g}{2}\right)\| x_t - x_{t+1}\|^2 +\left(\frac{1}{2\alpha } -\frac{\rho+1}{2} \right)\left\Vert \overline{x} -x_{t+1}\right\Vert ^{2}%\\&\qquad\qquad
    -\frac{1}{2\alpha }\|x_t-\overline{x}\|^2 - \frac{\alpha+1}{\alpha}\varepsilon  .\\
    \end{aligned}
    \end{equation}
% \begin{equation}
% \begin{aligned}
%     &\left(f+g\right)\left(\overline{x}\right)-\left(f+g\right)\left(x_{t+1}\right) \\
%     &\geq  -\frac{1}{2\alpha_t}\|x_t-\overline{x}\|^2 +\left(\frac{1}{2\alpha_t} -\frac{\rho}{2} - \frac{\eta_{t}}{2\alpha_{t}} \right) \left\Vert \overline{x} -x_{t+1}\right\Vert ^{2} - \varepsilon_t - \frac{\varepsilon_t}{\eta_{t}}.\\
%     \end{aligned}
%     \end{equation}
We then estimate the LHS of \eqref{eq:ineq_with_eta} by the assumption of sharpness on $f+g$
 \begin{align*}
   -\mu \dist(x_{t+1},S) &\geq   \left(\frac{1}{2\alpha } -\frac{\rho+1}{2} \right) \left\Vert \overline{x} -x_{t+1}\right\Vert ^{2} - \frac{1}{2\alpha }\|x_t-\overline{x}\|^2  +\left(\frac{1}{2\alpha } - \frac{L_g}{2}\right)\| x_t - x_{t+1}\|^2 - \frac{\alpha+1}{\alpha}\varepsilon.
 \end{align*}   
By using \eqref{eq:cond:alpha:1} from \Cref{assu:conditions} and the fact that $\overline{x}\in S$ can be taken arbitrarily, 
%first we further estimate the RHS by dropping the positive term $\left(\frac{1}{2\alpha } - \frac{L_g}{2}\right)\| x_t - x_{t+1}\|^2$. Then, 
we choose $\overline{x}\in S$ so that 
%we have 
 $$\|x_t - \overline{x}\|^2 \geq \dist^2(x_{t},S)\qquad \text{and} \qquad \|x_{t+1}- \overline{x}\|^2 \geq \dist^2(x_{t+1},S)$$
 which yields
 \begin{align*}
   -\mu \dist(x_{t+1},S) &\geq   \left(\frac{1}{2\alpha } -\frac{\rho+1}{2} \right) \dist^2(x_{t+1},S) - \frac{1}{2\alpha }\dist^2(x_{t},S) - \frac{\alpha+1}{\alpha}\varepsilon
 \end{align*}   
 %and
 % $\frac{1}{2\alpha }\|x_t-\overline{x}\|^2$ to the LHS of the inequality, we observe that both terms involving $\overline{x}\in S$ have a positive sign
 % \begin{equation*}
 % \begin{aligned}
 % &\frac{1}{2\alpha }\|x_t-\overline{x}\|^2 -\mu \dist(x_{t+1},S) \geq \\
 %  &\qquad\left(\frac{1}{2\alpha } -\frac{\rho}{2} - \frac{\eta }{2\alpha } \right) \left\Vert \overline{x} -x_{t+1}\right\Vert ^{2} +\left(\frac{1}{2\alpha } - \frac{L_g}{2}\right)\| x_t - x_{t+1}\|^2 - \varepsilon  - \frac{\varepsilon }{\eta }. \\
 %  \end{aligned}
 % \end{equation*} 
 % This allows us to take the infimum over $\overline{x}\in S$ in both sides of the inequality, yielding the following inequality involving the distance to the solution set
 % \begin{equation}
 % \label{eq:distance with eta and Lg}
 % \begin{aligned}
 %   &\frac{1}{2\alpha }\dist^2(x_t,S) -\mu \dist(x_{t+1},S) \geq \\ &\qquad \left(\frac{1}{2\alpha } -\frac{\rho}{2} - \frac{\eta }{2\alpha } \right) \dist^2(x_{t+1},S) +\left(\frac{1}{2\alpha } - \frac{L_g}{2}\right)\| x_t - x_{t+1}\|^2 - \varepsilon  - \frac{\varepsilon }{\eta }
 %   \end{aligned}
 % \end{equation}
% Rearranging the terms in the inequality above, we obtain
  \begin{equation}
  \label{eq:distance with eta and eps}
  \begin{split}
      \iff \dist^2(x_t,S) \geq \left(1 -\alpha  \rho - \alpha  \right) \dist^2(x_{t+1},S) +2\alpha  \mu \dist (x_{t+1},S) 
        -2(\alpha+1)\varepsilon
      % +(1-\alpha L_g) \Vert x_t -x_{t+1}\Vert^2 .
  \end{split}
 \end{equation}

Let now $t=t_0$. By hypothesis \eqref{eq:dist:xtplus}, we have 
\begin{equation}\label{eq:ineq:discriminant}
\begin{aligned}
\left(1-\alpha \rho-\alpha \right)\text{dist}^{2}\left(x_{t+1},S\right)+2\alpha \mu\text{dist}\left(x_{t+1},S\right)-2(\alpha+1)\varepsilon \leq (E^+)^2\\
% \leq\left(\frac{\alpha \mu+\sqrt{\alpha ^{2}\mu^{2}-2\alpha \left(\varepsilon +\frac{\varepsilon }{\eta }\right)\left(\alpha \rho+\eta \right)}}{\alpha \rho+\eta }\right)^{2}.
\end{aligned}
\end{equation}
which is a second degree inequality with respect  $\dist(x_{t+1},S)$.
he discriminant has the form %(see \Cref{sec:appendix:discriminant})
\begin{equation}
    \Delta =\left(\frac{\mu+\left(1-\alpha\left(\rho+1\right)\right)\sqrt{\mu^{2}-2\varepsilon\left(\rho+1\right)\frac{\alpha+1}{\alpha}} }{1+\rho}\right)^{2}  \geq0.
\end{equation}
Since $\Delta\geq 0$, the solutions are within the following range
\begin{align}\label{eq:discriminant}
\frac{-\alpha \mu-\sqrt{\Delta}}{1-\alpha \rho-\alpha } & \leq \dist \left(x_{t+1},S\right)\leq\frac{-\alpha \mu+\sqrt{\Delta}}{1-\alpha \rho-\alpha }  =\frac{\mu+\sqrt{\mu^{2}-2\varepsilon\left(\rho+1\right)\frac{\alpha+1}{\alpha}} }{1+\rho} = E^+
\end{align}
% {\color{purple} \textbf{remove} where the lower bound is clearly negative, while the upper bound reads as 
% \begin{equation}
% \begin{aligned}
% \frac{-\alpha \mu+\sqrt{\Delta}}{1-\alpha \rho-\eta } & = \frac{\alpha \mu+\sqrt{\alpha ^{2}\mu^{2}-2\alpha \left(\varepsilon +\frac{\varepsilon }{\eta }\right)\left(\alpha \rho+\eta \right)}}{\left(\alpha \rho+\eta \right)}
% \end{aligned}
% \end{equation}}
% (see \Cref{sec:appendix:discriminant}). 
In conclusion, since the distance from a set is a non-negative value, we obtain
\begin{equation}
0\leq \dist \left(x_{t+1},S\right)\leq
% \frac{\alpha \mu+\sqrt{\alpha ^{2}\mu^{2}-2\alpha \left(\varepsilon +\frac{\varepsilon }{\eta }\right)\left(\alpha \rho+\eta \right)}}{\left(\alpha \rho+\eta \right)}=
E^+.
\end{equation}
which by induction proves \eqref{eq:dist:xtplus}. 
The same reasoning can be applied to obtain \eqref{eq:dist:xtminus}. %: see \Cref{sec:appendix:discriminant}.
}
\end{proof}

\section{Proof of \Cref{prop:estimation with Et}}\label{sec:prop:estimation with Et}
\begin{proof}[Proof of \Cref{prop:estimation with Et}]
{ The quantities $E^+$ and $E^-$ are roots of the quadratic equation
\begin{equation}\label{eq:equation_Et}
    (E^-)^2 = (1-\alpha \rho -\alpha)(E^-)^2  +2\alpha \mu E^- -2(\alpha+1) \varepsilon .
\end{equation}
}
 Subtracting \eqref{eq:equation_Et} from \eqref{eq:distance with eta and eps},
 %(from the proof of \Cref{lem:inthetube:const}
 we obtain the following inequality that holds for every $t\in\N$
 \begin{equation}
\begin{aligned}
    &\operatorname{dist}^2(x_{t },S)-(E^-)^2\geq \\
    &  (1-\alpha \rho-\alpha)(\operatorname{dist}^2(x_{t +1},S)-(E^-)^2) +2\alpha \mu (\operatorname{dist}(x_{t +1},S)-E^-)+(1-\alpha L_g) \Vert x_t -x_{t+1}\Vert^2\\
    &= (1-\alpha \rho -\alpha+ \frac{2\alpha \mu}{\operatorname{dist}(x_{t +1},S)+ E^-})\left(\operatorname{dist}^2(x_{t +1},S)-(E^-)^2\right)+(1-\alpha L_g) \Vert x_t -x_{t+1}\Vert^2.
    \end{aligned}\end{equation}
    Then, 
    \begin{align} \label{eq:eps distance square estimate with xi}
    \operatorname{dist}^2(x_{t },S)-(E^-)^2 \geq \zeta_{{t }+1} \left(\operatorname{dist}^2(x_{t +1},S)-(E^-)^2\right)+(1-\alpha L_g) \Vert x_t -x_{t+1}\Vert^2,
\end{align}
where 
${\zeta_{t +1} := 1-\alpha \rho -\alpha+ \dfrac{2\alpha \mu}{\operatorname{dist}(x_{t +1},S)+ E^-}  \geq 0 }$
is  non-negative  by \Cref{assu:conditions}.
By dropping the term $(1-\alpha L_g) \Vert x_t -x_{t+1}\Vert^2$ (positive by \eqref{eq:cond:alpha:1}), we have
\begin{align} \label{eq:eps distance square estimate with xi 2}
    \operatorname{dist}^2(x_{t },S)-(E^-)^2 \geq \zeta_{{t }+1} \left(\operatorname{dist}^2(x_{t +1},S)-(E^-)^2\right).
\end{align}
Now we show that at iteration $t\geq t_0$, we have $\zeta_{t+1}\geq 1$. Notice that  $\zeta_{t+1}\geq 1$  whenever 
%\begin{align}
    $\frac{2\alpha \mu}{\operatorname{dist}(x_{t +1},S)+ E^-} \geq \alpha \rho +\alpha.$
    %\end{align}
    Equivalently,
    \begin{align}
    \operatorname{dist}(x_{t +1},S) & \leq \frac{2 \mu}{\rho + 1} - E^-  =\frac{\mu +\sqrt{\mu^{2}-2\varepsilon(\rho+1)\frac{\alpha+1}{\alpha} }}{\rho+1} = E^+.
\end{align}
{By assumption, there exists $t_0$ such that $\dist(x_{t_0},S)\leq E^+$ and, by \Cref{lem:inthetube:const}, this implies that \linebreak ${\dist(x_{t+1},S)\leq E^+}$. In conclusion,   for every $t\geq t_0$ the condition $\zeta_{t+1} \geq 1$ holds.
}
\end{proof}

\section{Proof of sharpness for $\bar{f}$ in \Cref{rmk: prox form of x2 -1}} \label{app:proof_of_lemma_sharp}
\begin{proof}
 For  $x\in\X$, the distance to the sphere $\mathcal{S}$  is given as ${x}\mapsto{|\|x\|-1 |}$.
     We show that for $\mu=1$ we have
\begin{equation*}
 (\forall x \in \X)\qquad    |\|x\|^2 - 1| \geq \mu \dist(x,S). 
    \end{equation*}
% By applying the inverse triangle inequality we can further estimate the RHS as
% \begin{equation}
% \begin{aligned}
%     |\|x\|^2 - 1| \geq \mu \|x - s\| \geq \mu |  \|x\| - \|s\|| = \mu |  \|x\| - 1|
%     \end{aligned}
% \end{equation}
If $x=0$, then $\dist(x,\mathcal{S}) = 1$ and the inequality is satisfied for $\mu=1$. Otherwise, for every $x\in\X\setminus 0$ we have
\begin{equation*}
\min_{s\in \mathcal{S}} \|x-s\|  = 
    \|x- \frac{x}{\|x\|}\| = \| x \left(1 - \frac{1}{\|x\|}\right)\| = \|x\| \left| 1- \frac{1}{\|x\|}\right| = \|x\| \frac{1}{\|x\|} \left | \|x\| - 1\right| = |\|x\|-1|.
 \end{equation*}
% where the third equality stems from the absolute homogeneity of a norm.
Consider $ \fonc{\tilde f}{\R}{\R}$ defined as ${ \sigma \mapsto |\sigma^2-1|}$, \emph{i.e.\ }  the function from \Cref{ex: weak convexity} with  $(a,b)=(-1,1)$. The following inequality is satisfied for every $\delta \leq 1$ (as illustrated in \Cref{fig:comparisons})
\begin{equation*}
(\forall \sigma \in\R)  \qquad \delta | |\sigma|-1| = \delta \min\{ |\sigma - 1|,  |\sigma +1| \}\leq  |\sigma^2 -1 |.
\end{equation*}
where the last inequality holds by the sharpness of function $\tilde f$ (see \Cref{ex:sharpness}). Choosing $\sigma = \|x\|$ we obtain
\begin{equation*}
 (\forall x \in \X\setminus 0)\qquad   \delta| \|x\| - 1| \leq   |\|x\|^2 - 1|. 
\end{equation*}
Thus,   $\mu\dist(x,\mathcal{S})\leq   f(x)$, for  $\mu \in(0, 1]$, for $x \in \X$
which implies the global sharpness of  $f$ with  $\mu=1$.
\end{proof}

\section{Inexact Proximal Computation}\label{app:inexprox}
{

%To compute inexact proximal points, we rely on the strategy based on the definition of \emph{$\varepsilon$-proximal points} 
%as \emph{$\varepsilon$-solutions} of the corresponding optimisation problem 
%(see \Cref{def:eps_prox}). \\

For any strongly convex function $h$ and its unique minimiser $\overline{x}$,  
%the  level set
\begin{equation}
    \operatorname{lev}_{\leq \varepsilon + h(\overline{x})} h : = \{x\in\X\,|\, h(x) \leq \varepsilon + h(\overline{x})\}
\end{equation}
is the set of  $\varepsilon$-solutions of $h$.
%, $\varepsilon\geq 0$.
\sloppy If, for some $y\in\X$, and $f$ $\rho$-weakly convex with $\alpha<\frac{1}{\rho}$, for every $x\in\X$, function $h$ is defined as 
 ${h(x): = f(x) + \frac{1}{2\alpha}\|x-y\|^2}$,
then
$
    \varepsilon\operatorname{-prox}_{\alpha f}(y) =  \operatorname{lev}_{\leq \varepsilon + h(\overline{x})}h$,
where $\overline{x}$ is the exact proximal point of $f$ at $y$.

\begin{lemma}
Let $\fonc{h,\,\widehat{h}}{\X}{(-\infty,+\infty]}$ be two strongly convex functions with minimisers $\overline{x}\in\X$ and $\widehat{x}\in\X$ respectively. Assume that for every $\varepsilon\geq 0$ 
\begin{equation}\label{eq:ineq}
    (\forall x \in \X)\qquad \qquad h(x) \leq \widehat{h}(x) \leq h(x) + \varepsilon.
\end{equation}
Then, the unique minimiser $\widehat{x}$ of $\widehat{h}$ belongs to $ \operatorname{lev}_{\leq \varepsilon + h(\overline{x})}h$.\\
\end{lemma}
\begin{proof}
By contradiction,  assume that $\widehat{x}\notin \operatorname{lev}_{\leq \varepsilon + h(\overline{x})}h$, \emph{i.e.\ }, $h(\widehat{x}) > \varepsilon  + h(\overline{x}).$ By assumption,  for $x=\widehat{x}$, equation \eqref{eq:ineq} reads as
$
    h(\widehat{x}) \leq \widehat{h}(\widehat{x}) \leq h(\widehat{x}) + \varepsilon
$
which combined with our hypothesis implies
$
    \varepsilon + h(\overline{x}) < \widehat{h}(\widehat{x}).
$
On the other side, for $x=\overline{x}$, equation \eqref{eq:ineq} reads as
$
    h(\overline{x}) \leq \widehat{h}(\overline{x}) \leq h(\overline{x}) + \varepsilon
$,
which implies $
     \widehat{h}(\overline{x}) \leq \varepsilon + h(\overline{x}).$ This induces
    $ \widehat{h}(\overline{x}) \leq \varepsilon + h(\overline{x}) < \widehat{h}(\widehat{x})$
which is a contradition because $\widehat{x}$ is the unique minimiser of $\widehat{h}$.
\end{proof}

\begin{example}
    Let 
    %us consider function $f$ from \Cref{ex: weak convexity} 
     $(a_1,\dots,a_n), (b_1,\dots,b_n)\in\R^n$ be  pair of vectors. The function $\fonc{F}{\R}{\R}$ defined, for every $x\in\R^n$, as
$F(x) := \sum_{i=1}^n |(x-a_i)(x-b_i)|$ 
lacks an explicit formula for its proximal operator. Consider the function $\fonc{\widehat{F}}{\R^n}{\R}$ defined, for every $x\in\R^n$ as 
$\widehat{F}(x) = \sum_{i=1}^n \sqrt{((x-a_i)(x-b_i))^2 + \varepsilon_i^2} $
for given $\varepsilon_i>0$, $i=1,\dots,n$. It is  easy to see that for every $x\in\R^n$
\begin{equation}
    F(x) \leq \widehat{F}(x) \leq \sum_{i=1}^n \left(|((x-a_i)(x-b_i))| + \varepsilon_i\right) \leq F(x) + \sum_{i=1}^n \varepsilon_i.
\end{equation}
Functions $h:=F + \frac{1}{2\alpha}\|\cdot - y\|^2$ and $\widehat{h}: =\widehat{F}  + \frac{1}{2\alpha}\|\cdot - y\|^2$ satisfy \eqref{eq:ineq} with $\varepsilon=\sum_{i=1}^n\varepsilon_i$ and by solving the surrogate smooth and strongly convex problem
\begin{equation}\label{eq:surrogate_smooth_problem}
    \argmin_{x\in\R^n} \widehat F(x) + \frac{1}{2\alpha}\|x-y\|^2
\end{equation}
we obtain a point in $\varepsilon\operatorname{-prox}_{\alpha F}(y)$. In particular, solving \eqref{eq:surrogate_smooth_problem} corresponds to solving $n$ independent smooth and strongly convex 1-dimensional problems.
\end{example}}
\end{appendices}

\end{document}